\documentclass[a4paper,oneside,9pt]{article}
\usepackage{blindtext}
\usepackage{authblk}
\usepackage{newclude}
\usepackage{graphicx}
\usepackage{float}
\usepackage{wrapfig}
\usepackage[utf8]{inputenc}
\usepackage{amsmath}
\usepackage{hyperref}
\usepackage{amsfonts, amssymb, amsthm}
\usepackage{tikz}
\usepackage{scalerel}
\usetikzlibrary{angles,patterns,calc}
\usepackage{mathtools}
\usepackage{stmaryrd}
\usepackage[figure]{hypcap}
\usepackage{etoolbox}
\usepackage{enumerate}
\usepackage{bbm}
\usepackage{esint}
\usepackage{xurl}
\usepackage[a4paper, left=2cm, right=2cm, top=2cm, bottom=2cm]{geometry}
\usepackage[sorting=nty,citestyle=numeric-comp,bibstyle=ieee,maxbibnames=99, backend=bibtex, url=false, isbn=false, date=year, doi=false,eprint=true, natbib=true]{biblatex}
\usepackage{fancyhdr}
\usepackage{yhmath}

\newcommand{\R}{\mathbb{R}}

\newcommand{\M}{\mathcal{M}}
\newcommand{\norm}[1]{\left\|{#1}\right\|}
\newcommand{\abs}[1]{\left|{#1}\right|}
\newcommand{\dc}[1]{\mathopen{\scalebox{1.0}{$\lbrace\!\mkern-1mu\lbrace$}}#1\mathclose{\scalebox{1.0}{$\rbrace\!\mkern-1mu\rbrace$}}}
\newcommand{\jump}[1]{\ensuremath{\left\llbracket #1\right\rrbracket}}
\newcommand{\keywords}[1]{\textbf{Keywords:} #1}

\newtheorem{defn}{Definition}[section]
\newtheorem{thm}[defn]{Theorem}

\newtheorem{lemma}[defn]{Lemma}

\newtheorem{prop}[defn]{Proposition}
\newtheorem{remark}[defn]{Remark}
\newtheorem{ass}[defn]{Assumption}
\newtheorem{cond}[defn]{Condition}
\title{A posteriori error control for a finite volume scheme for a cross-diffusion model of ion transport}
\date{}
\begin{document}
	\author[1]{Arne Berrens\thanks{Corresponding author; berrens@mathematik.tu-darmstadt.de}}
	\author[1]{Jan Giesselmann\thanks{giesselmann@mathematik.tu-darmstadt.de}}
	\affil{\centering Department of Mathematics\\ Technical University of Darmstadt\\ Dolivostr. 15, 64293 Darmstadt, Germany}
	\maketitle
\begin{abstract}
	We derive a reliable a posteriori error estimate for a cell-centered finite volume scheme approximating a cross-diffusion system modeling ion transport through nanopores.
	To this end, we derive a stability framework that is independent of the numerical scheme and introduce a suitable (conforming) reconstruction of the numerical solution.
	The stability framework relies on some simplifying assumptions that coincide with those made in weak uniqueness results for this system.
	Additionally, when electrical forces are present, we assume that the solvent concentration is uniformly bounded from below.
	This is the first a posteriori error estimate for a cross-diffusion system.
	Along the way, we derive a pointwise a posteriori error estimate for a finite volume scheme that approximates the diffusion equation.
	We conduct numerical experiments showing that the error estimator scales with the same order as the true error.
\end{abstract}
\section{Introduction}
\subsection{Ion Transport Model}
We consider the ion transport model presented in \textcite{burger2012}.
Similar to the Poisson-Nernst-Planck equations given by \textcite{nernst1888}, it models ionic species in a electrically neutral solvent.
In contrast to the Poisson-Nernst-Planck equations, the ion transport model incorporates forces due to finite size of every particle.
In essence, this means that the transport of ions is affected by size exclusion effects.
This comes into play when modeling ion transport through nanopores or membranes, i.e. when the size of an individual ion is not negligible compared to the channel size.
Due to size exclusion, different ionic species influence their diffusion and Fick's law is no longer true in small geometries (see \cite{simpson2009,burger2012,massimini2024} and references therein).
Therefore, the ion transport model incorporates cross-diffusion terms other than the classical Poisson-Nernst-Planck equations.
Furthermore, the ion transport model incorporates the solvent concentration and the convection with it, i.e. the ion particles are dragged along with the movement of the solvent.
Incorporating the solvent concentration leads to the volume filling condition, namely that the volume fraction of the solvent is the remaining volume that is not occupied by the ions.
This is modeled in the second equation in \eqref{equations}.\\
In general, cross-diffusion systems are analytically and numerically challenging, since classical methods, like maximum or comparison principles are not applicable (see \cite{jungel2016,stara1995, massimini2024}).\\
The ion transport model includes the ion concentrations $(v_1,\dots,v_n)$, solvent concentration $v_0$ and electric potential $\Psi$ in a Lipschitz domain $\Omega\subset\R^d$ for $d=2$ or $d=3$.
The ion and solvent concentrations satisfy the equations for some $T>0$
	\begin{align}\label{equations}
		\begin{aligned}
			\partial_t v_i &=  D_i\text{div}(v_0\nabla v_i-v_i\nabla v_0+v_0v_i\beta z_i\nabla \Psi)\quad i=1,\dots,n\\
			v_{0} &= 1-\sum_{i=1}^n v_i,
		\end{aligned}
	\quad\text{in }\Omega\times(0,T)
	\end{align}
where $D_1,\dots,D_n>0$ are the diffusion coefficients, $\beta>0$ is the inverse thermal voltage and $z_i\in\R$ the charge of each particle of the $i$th species.
The electric potential satisfies
\begin{align}\label{eqn:phi}
	-\lambda^2\Delta\Psi = \sum_{i=1}^n z_iv_i+f \quad\text{in }\Omega\times(0,T),
\end{align}
where $\lambda>0$ is the permittivity, $f\in L^\infty(\Omega)$ is the permanent charge density.
We fix a decomposition $\Gamma_D,\Gamma_N$ of $\partial\Omega$ and prescribe
\begin{align}\label{boundaryCond}
	\begin{aligned}
		\left(v_0\nabla v_i-v_i\nabla v_0+v_0v_i\beta z_i\nabla\Psi\right)\cdot n &= 0 \text{ on } \Gamma_N\times[0,T],\quad v_i = v_i^D \text{ on } \Gamma_D\times[0,T],\quad i=1,\dots,n\\
		\nabla\Psi\cdot n &= 0 \text{ on } \Gamma_N\times[0,T],\quad\, \Psi = \Psi^D\text{ on } \Gamma_D\times[0,T]
	\end{aligned}
\end{align}
and consider
\begin{align}\label{eqn:initial}
	v_i(\cdot,0) = v_i^0\quad\text{in }\Omega,\; i=0,\dots,n
\end{align}
for given functions $v_i^D,\Psi^D\in L^2(0,T;L^2(\Gamma_D))$ and $v_i^0\in L^\infty(\Omega)$.
We further assume, that $v_i^D$ and $\Psi^D$ are traces of functions in $L^2(0,T;L^\infty(\Omega))\cap L^2(0,T;H^1(\Omega))$.
We will discuss the regularity of the solution in Theorem \ref{thm:existence}.\\
The main goal of this work is to establish an a posteriori error estimate for a finite volume scheme approximating the system \eqref{equations}-\eqref{eqn:initial}.
In the case of no electrical forces, i.e. $z_i=0$ for $i=1,\dots,n$, we derive an a posteriori error estimate under the assumption of constant diffusion constants, namely $D_i=D$ for $i=1,\dots,n$ and some conditions on the data (see Condition \ref{conditions} (C1)-(C7)).
In the presence of electrical forces, we further need that the solvent concentration is uniformly bounded from below (see Assumption \ref{assumptions}) as well as constant diffusion coefficients and equal charges and some conditions on the data (see Condition \ref{conditions} (C1)-(C5)).
This is a first step towards a posteriori error analysis of cross diffusion systems.
To our knowledge, there are no previous a posteriori error estimates for finite volume schemes approximating general non-linear parabolic systems with positive diffusion matrix.
\textcite{cochez2009} proved an a posteriori error estimate for a finite volume method approximating a non-linear steady-state diffusion problem.
Numerical schemes for cross-diffusion systems have been develop and studied in \cite{barrett2004,burger2020,jungel2021} and specifically for the ion transport model \cite{cances2019, gerstenmayer2018, gerstenmayerFiniteelement2018,cances2019}.
In \cite{gerstenmayerFiniteelement2018} the a priori convergence for a finite element scheme is shown and a priori convergence of a finite volume scheme is proven in \cite{cances2019} for a reduced model and in \cite{massimini2024} for the full model.
In this paper, we initiate the previously missing a posteriori analysis.
We use a cell centered finite volume scheme found in \cite{gerstenmayer2018}.
Finite volume schemes are commonly used for cross-diffusion systems, since it is easier to prove some important properties of the resulting numerical approximation like non-negativity and entropy dissipation compared to finite element schemes.
Therefore, we consider a finite volume scheme although it is usually harder to derive a posteriori error estimates for finite volume schemes than for finite element schemes.\\
In contrast, there are a posteriori error estimates for finite element schemes for non-linear parabolic systems see \cite{verfurth1998,sutton2020} and even ones allowing for degenerate diffusion coefficient in \cite{cances2020}.
In \cite{sutton2020} nonlinear diffusion systems are studied with symmetric and positive definite diffusion matrices that do only depend on space and time and not on the species $v_1,\dots,v_n$.
In \cite{verfurth1998} more general nonlinear parabolic systems are considered that also allow for species dependent diffusion matrices that are still symmetric and positive definite.
\textcite{cances2020} provide a posteriori error estimates for a finite element discretization of the nonlinear anisotropic Fokker-Planck equation for only one species.
\subsection{The Model}
The existence of weak solutions for the stationary version of \eqref{equations}-\eqref{eqn:phi} was shown in \cite{burger2012}.
In \cite{zamponi2015} for the time dependent problem with pure Neumann boundary conditions and in \cite{gerstenmayer2018} for mixed Neumann and Dirichlet boundary conditions.
Uniqueness of weak solutions was established in \cite{zamponi2015} under the assumption of equal diffusion coefficients $D=D_i$ for $i=1,\dots,n$ and without electrical forces, i.e. $\Psi=0$.
In \cite{gerstenmayer2018}, the uniqueness of weak solutions with electrical forces was shown under the assumption of equal diffusion coefficients $D=D_i$ for $i=1,\dots,n$ and of equal charge $z=z_i$ for $i=1,\dots,n$.\\
In \cite{jungel2016,massimini2024}, System \eqref{equations} is written in the form
\begin{align*}
	\partial_t v_i = \operatorname{div}\left(\sum_{j=1}^n\left(A_{i,j}(v)\nabla v_j+Q_{i,j}(v)\nabla\Psi\right)\right),\quad i=1,\dots,n
\end{align*}
with $v = (v_1,\dots,v_n)$ and $A(v),Q(v)\in\R^{n\times n}$
\begin{align*}
	(A(v))_{i,j} = 
	\begin{cases}
		D_i\left(1-\sum_{\substack{k=1\\i\neq k}}^n v_k\right) &\text{if } i=j\\
		D_iv_i &\text{if } i\neq j
	\end{cases},\quad Q_{i,j}(v)=v_0v_i\delta_{i,j}.
\end{align*}
The matrix $A(v)$ is called the diffusion matrix.
\textcite{gerstenmayer2018} proved the global existence of a weak solution by a compactness argument via the entropy variables.
In what follows, we write $\abs{\cdot}$ for the $d$-dimensional or $d-1$-dimensional Lebesgue measure.
For $\abs{\Gamma_D}>0$ we define the spaces
\begin{align*}
	H_D^1(\Omega) := \{u\in H^1(\Omega)\,:\,u = 0\quad\text{on }\Gamma_D\}
\end{align*}
and $H^1_D(\Omega)'$ the dual space of $H^1_D(\Omega)$.
We equip the space $H_D^1(\Omega)$ with the norm
\begin{align*}
	\norm{\cdot}_{H^1_D(\Omega)} = \norm{\nabla \cdot}_{L^2(\Omega)}.
\end{align*}
This norm is equivalent to the classical $H^1$ norm on $H^1_D$ via the general Poincar\'e-Friedrichs inequality, Theorem \ref{thm:friedirchsineq}.
We denote by $\langle\cdot,\cdot\rangle$ the duality pairing of $H^1_D(\Omega)'$ and $H^1_D(\Omega)$.
The next theorem states the existence of a global bounded weak solutions to \eqref{equations}-\eqref{eqn:initial} under suitable conditions.
\begin{thm}{\cite[Theorem 1]{gerstenmayer2018}}\label{thm:existence}
	Let $\Gamma_N\cap\Gamma_D=\emptyset$ with $\Gamma_N$ open and $|\Gamma_D|>0$.
	For all $T>0$, there exists a global bounded weak solution $v_0,\dots,v_n:(0,T)\times\Omega\to[0,1], \Psi:(0,T)\times\Omega\to\R$ of \eqref{equations}-\eqref{eqn:initial} with $\sum_{i=0}^nv_i=1$ such that
	\begin{align*}
		v_i\sqrt{v_0},\sqrt{v_0}&\in L^2(0,T;H^1(\Omega)),\quad \partial_t v_i\in L^2(0,T;H^1_D(\Omega)')\quad i=1,\dots,n,\\
		\Psi&\in L^2(0,T;H^1(\Omega))
	\end{align*}
	and
	\begin{align*}
		\int_0^T \langle\partial_t v_i,\phi_i\rangle dt+D_i\int_0^T\int_\Omega \left(v_0\nabla v_i-v_i\left(\nabla v_0+v_0\beta z_i\nabla\Psi\right)\right)\cdot\nabla \phi_i dxdt &= 0\\
		\lambda^2\int_0^T\int_\Omega \nabla \Psi\cdot\nabla\theta dxdt - \int_0^T\int_\Omega\left(\sum_{i=1}^n z_iv_i+f\right)\theta dxdt &=0
	\end{align*}
	for all $\phi_i,\theta\in L^2(0,T;H^1_D(\Omega))$, $i=1,\dots,n$.
	The boundary conditions \eqref{boundaryCond} are fulfilled in the sense of traces in $L^2(\Gamma_D)$ and the initial conditions \eqref{eqn:initial} are satisfied in the sense of $H^1_D(\Omega)'$.
\end{thm}
We assume later on (see Assumptions \ref{assumptions}), that $v_0$ is bounded away from $0$.
In this case, according to \cite[Remark 4]{gerstenmayer2018}, $v_i\in L^2(0,T;H^1(\Omega))$ for all $i=0,\dots,n$.\\
Similar to the uniqueness proof in \cite{gerstenmayer2018} and the convergence proof of \cite{cances2019}, we need the assumption $D_i=D$ and $z_i=z$ for $i=1,\dots,n$ and further assume for simplicity $D=1$ with the same boundary \eqref{eqn:initial} and initial \eqref{boundaryCond} conditions (see Condition \ref{conditions} (C5)).
In this case the equations
\begin{align}\label{eqn:general}
	\begin{split}
		\partial_t v_i &= \text{div}(v_0\nabla v_i-v_i\nabla v_0+v_0v_i\beta z\nabla \Psi)\quad i=1,\dots,n\\
		v_{0} &= 1-\sum_{i=1}^n v_i\\
		-\lambda^2\Delta\Psi &= z(1-v_0)
	\end{split}\quad\text{in }\Omega\times(0,T)
\end{align}
decouple in the sense that the solvent concentration $v_0$ solves
\begin{align}\label{eqn:solvent}
	\partial_t v_0=\operatorname{div}(\nabla v_0-v_0(1-v_0)z\beta\nabla \Psi).
\end{align}
Equation (6) is obtained by summing up the evolution equations for the $v_i$ and using the fact that the sum of all concentrations is $1$.
We call \eqref{eqn:general} the general model for convenience.
Similar to \cite{cances2019,jungel2016}, we first consider the reduced model
\begin{align}\label{eqn:reduced}
	\begin{split}
		\partial_t v_i&=\operatorname{div}(v_0\nabla v_i-v_i\nabla v_0)\quad i=1,\dots,n\\
		v_0 &= 1-\sum_{i=1}^nv_i
	\end{split}\quad\text{in }\Omega\times(0,T)
\end{align}
which is obtained by setting $z=0$.
In the reduced model, the solvent concentration solves the diffusion equation
\begin{align}\label{eqn:solvent_reduced}
	\partial_t v_0 = \text{div}(\nabla v_0) = \Delta v_0\quad\text{in }\Omega\times(0,T).
\end{align}
\subsection{Main results and proof strategy}
\subsubsection{Main results}
We study a cell centered finite volume scheme for the system \eqref{eqn:general} and derive an a posteriori error estimate. 
For this we need a conforming reconstruction of the finite volume solution denoted by $(\widehat{u_0},\dots,\widehat{u_n},\widehat{\Phi})$.
Let $(v_0,\dots,v_n,\Psi)$ be a weak solution to the system \eqref{eqn:general} in the sense of Theorem \ref{thm:existence}.
We list all conditions on the data in Condition \ref{conditions} and assumptions in Assumption \ref{assumptions} that are needed later.
We do not need all conditions and assumptions at every time and specify the needed conditions and assumptions to prove the statement.
\begin{cond}\label{conditions}\hfill
	\begin{enumerate}[(C1)]
		\item The domain $\Omega\subset\R^d$ is an open, bounded domain with polygonal boundary.
			Furthermore, let $\Gamma_D,\Gamma_N$ be disjoint with $\partial\Omega=\Gamma_D\cup\Gamma_N$, $\Gamma_N$ open and $\abs{\Gamma_D}>0$.
		\item	Let $v_0^0,\dots,v_n^0\in L^\infty(\Omega)\cap H^1(\Omega)$ with $v_i^0>0$ for all $i=1,\dots,n$ and $\sum_{i=0}^n v_i^0 = 1$ on $\Omega$.
		\item Let the boundary conditions $v_{i}^D$ $i=0,\dots,n$ be constant in time and constant on connected components of $\Gamma_D$.
		\item There exists $\gamma>0$ such that 
			\begin{align*}
				\gamma &< v_0^0(x) \leq 1\quad\forall x\in\Omega.
			\end{align*}
		\item The diffusion coefficients and charges are all equal, i.e. $D_i=D$ and $z_i=z$ for $i=1,\dots,n$.
		\item There exists a uniform bound of the $W^{1,p^*}_D(\Omega)$-norm for some $1\leq p^*\leq \frac{d}{d-1}$ of the Laplace Green's function with homogeneous Dirichlet conditions on $\Gamma_D$ and homogeneous Neumann conditions on $\Gamma_N$ ,i.e.
			\begin{align*}
				\norm{\nabla_x G(x;y)}_{L^{p^*}(\Omega)}\leq C_{\text{Green},p^*}\quad\forall y\in\Omega.
			\end{align*}
		\item There exists a harmonic continuation $u^D_0\in W^{1,q}(\Omega)$ of $u^D_0\in W^{1,q}(\Gamma_D)$ with $\nabla u^D_0\cdot n=0$ on $\Gamma_N$.
	\end{enumerate}
\end{cond}
\begin{ass}\label{assumptions}
	Assume that there exists $\gamma>0$ such that
	\begin{align*}
		\gamma &< v_0(x,t)\leq1\quad\forall x\in\Omega\text{ and }t\in[0,T].
	\end{align*}
\end{ass}
\begin{remark}\label{remark:ass}
	\begin{enumerate}[(i)]
		\item Conditions (C1) and (C2) guarantee that there exists a weak solution to \eqref{equations}-\eqref{eqn:initial} according to Theorem \ref{thm:existence}.
		\item Condition (C3) ensures that the Dirichlet condition is fulfilled exactly by the numerical reconstruction.
		In most applications this assumption is not a severe restriction, since the Dirichlet boundary conditions are constant on connected subsets that model connections to comparably large reservoirs with constant concentration.
			We discuss Assumption (C3) in more detail in Section \ref{section:abstract_error}.
		\item Condition (C4) ensures that the solvent concentration in the exact solution $v_0$ is always bounded from below by $\gamma$ in the case of $z=0$.
			To ensure that the solvent concentration is bounded from below in the case $z\neq 0$ we need Assumption \ref{assumptions}.
			This is important, because equation \eqref{equations} degenerates for $v_0=0$ and we do not get a bound on the error of the gradient in this case.
			However, $v_0$ models the solvent concentration and therefore it is physically meaningful to assume that $v_0$ is bounded away from $0$.
		\item Condition (C5) ensures, that we obtain an equation for the solvent that is independent of all other species.
			This condition is also used in the uniqueness proof in \cite{gerstenmayer2018}.
		\item Conditions (C6) and (C7) are only needed for the $L^\infty$-norm error estimate for the diffusion equation from Section \ref{section:heat}.
			There are certain cases where \ref{conditions} (C6) can be proven to hold.
			These are
			\begin{itemize}
				\item for $d\geq 2$, $\Gamma_D=\partial\Omega$ and $1\leq p^*<\frac{d}{d-1}$. The proof for the case $d=2$ can be found in \cite[Theorem 2.12, Remark 2.19]{dong2009} and for $d>2$ in \cite[Theorem 1.1]{gurter1982}.
				\item for $d\geq 2$, $\Gamma_N=\partial\Omega$ and $1\leq p^*<\frac{d}{d-1}$. The proof can be found in \cite[Theorem 1]{mitrea2010}.
			\end{itemize}
			Usually, the constant $C_{\text{Green},p^*}$ is not accessible.
	\end{enumerate}
\end{remark}
The next theorem states the a posteriori error estimate for the general model \eqref{eqn:general}.
The constants and estimators showing up in the next theorem are made explicit in Section \ref{section:residual}.
\begin{thm}\label{thm:main_estimator}
	Let the data satisfy Condition \ref{conditions} (C1)-(C5) and let the weak solution $(v_0,\dots,v_n,\Psi)$ of \eqref{eqn:general} with initial and boundary conditions \eqref{boundaryCond}-\eqref{eqn:initial} in the sense of Theorem \ref{thm:existence} satisfy Assumption \ref{assumptions}.
	Let $(\widehat{u_0},\dots,\widehat{u_n},\widehat{\Phi})$ be the reconstruction of a finite volume solution defined in Section \ref{section:scheme}.
	\begin{enumerate}[(i)]
		\item The difference $\widehat{u_0}-v_0$ satisfies
			\begin{multline*}
				\max_{t\in[0,T]}\norm{\widehat{u_0}(t,\cdot)-v_0(t,\cdot)}_{L^2(\Omega)}^2+\norm{\nabla (\widehat{u_0}-v_0)}_{L^2([0,T]\times\Omega)}^2\\ 
				\leq\left(2\norm{\widehat{u_0}^0-v_0^0}_{L^2(\Omega)}^2+\sum_{j=0}^J\tau_j\left(\frac{\abs{z\beta}^2}{4}\left(\eta_{R,\Phi}^j\right)^2+4\left(\eta_{R,0}^j\right)^2\right)\right)\\
				\times\exp\left(2C_G^{\frac{2}{1-\theta}}(1+C_{F,2,\Gamma_D}^2)^{\frac{\theta}{1-\theta}}\mu\norm{\nabla \widehat{\Phi}}_{X(q)}^{\frac{2}{1-\theta}}+C_{F,2,\Gamma_D}^2\frac{\abs{z}^4\beta^2}{8\lambda^4}T\right)=:\eta_2^J.
			\end{multline*}
		\item The $L^2$-in-time, $H^1$-in-space seminorm of $\widehat{\Phi}-\Psi$ is bounded by
			\begin{align*}
				\norm{\nabla (\widehat{\Phi}-\Psi)}_{L^2([0,T]\times\Omega)}^2 
				\leq C_{F,2,\Gamma_D}^2\frac{\abs{z}}{\lambda^2}\eta_2^J+\sum_{j=0}^J \tau_j\left(\eta_{R,\Phi}^j\right)^2.
			\end{align*}
		\item The error $\widehat{u_i}-v_i$ satisfies
			\begin{multline}
				\max_{t\in[0,T]} \norm{\widehat{u_i}-v_i}_{L^2(\Omega)}^2+\norm{\sqrt{v_0}\nabla(\widehat{u_i}-v_i)}_{L^2([0,T]\times\Omega)}^2
				\leq
				\left(2\norm{\widehat{u_i}^0-v_i^0}^2_{L^2(\Omega)}\right.\\
				+\frac{12}\gamma\eta_2^J\left(\norm{\nabla \widehat{u_i}}_{L^\infty(0,T;L^{\tilde{q}}(\Omega))}^2+2\left((1+C_S\sqrt{1+C_{F,2,\Gamma_D}})\abs{z\beta}\norm{\nabla \widehat{\Phi}}_{L^\infty(0,T;L^{\tilde{q}}(\Omega))}+C_{F,2,\Gamma_D}\frac{z^2\beta}{\lambda^2}\right)^2\right)\\
				\left.+\frac{12}\gamma C_{F,2,\Gamma_D}^2\sum_{j=0}^J\tau_j\left(\abs{z\beta}^2\left(\eta_{R,\Phi}^j\right)^2+\left(\eta_{R,i}^j\right)^2\right) \right)
				\exp\left(2C_G^{\frac{2}{1-\theta}}(1+C_{F,2,\Gamma_D}^2)^{\frac{\theta}{1-\theta}}\frac{\mu}{\gamma^{\frac{1+\theta}{1-\theta}}}\norm{F}_{X(q)}^{\frac{2}{1-\theta}}\right).
			\end{multline}
	\end{enumerate}
	With $F:=\nabla \widehat{u_0}-\widehat{u_0}z\beta\nabla\widehat{\Phi}$, the constants $\theta = \frac d2-\frac dp$ and $\mu= \frac{1-\theta}{2}\left(\frac{1}{2(1+\theta)}\right)^\frac{\theta+1}{\theta-1}$.
\end{thm}
We also derive an a posteriori error estimate for the reduced model \eqref{eqn:reduced} in Theorem \ref{thm:error_estimator}.
For this, we first consider the solvent concentration $v_0$ that solves the diffusion equation \eqref{eqn:solvent_reduced}.
We establish a posteriori error estimates for $\widehat{u_0}-v_0$ in the $L^\infty([0,T]\times\Omega)$ norm under a further assumption on the domain $\Omega$ and the boundary conditions (see Condition \ref{conditions} (C6) and (C7)) and for $\nabla (\widehat{u_0}-v_0)$ in the $L^2([0,T]\times\Omega)$ norm.\\
For the maximum norm error we use a technique similar to \cite{kyza2018,demlow2012,demlow2009}.
Such a maximum norm a posteriori error estimate is new for any finite volume scheme for the diffusion equation \eqref{eqn:solvent_reduced} and we think it is interesting in its own right.
Similar estimates have been established for various finite element schemes see \cite{demlow2009,demlow2012}.\\
\subsubsection{Proof strategy}
We derive a residual based a posteriori error estimator using the approach presented in \cite{makridakis2003,makridakis2007}.
First, we derive a stability framework for the reduced, Theorem \ref{thm:abstract_error}, and full, Theorem \ref{thm:abstract_general}, ion transport models.
In essence, we provide upper bounds on the norm of the difference of the exact solution and the solution to a perturbed system, where the upper bound only depends on the solution to the perturbed system and the "perturbation".
To derive an a posteriori error estimate, we interpret approximate solutions as solution to a perturbed system, where the residuals serve as the perturbation.
Hence, once the stability framework has been derived, we only need to bound the residuals in a suitable norm that matches the stability framework to obtain an a posteriori error estimate.
\\
We state stability frameworks for the ion concentrations first for the reduced model \eqref{eqn:reduced} and then for the general model \eqref{eqn:general}.
Both are independent of the numerical scheme and reconstruction used.
We assume that $(u_0,\dots,u_n,\Phi)$ is a weak solution of the perturbed system 
\begin{align}\label{eqn:fullperturbed}
	\begin{aligned}
	\partial_t u_i-\operatorname{div}(u_0\nabla u_i-u_i\nabla u_0+u_0u_i\beta z\nabla\Phi)&=R_i\quad i=1,\dots,n\\
	\sum_{i=0}^n u_i&= 1\\
	-\lambda^2\Delta\Phi -z(1-u_0)&=R_\Phi
	\end{aligned}
	\quad\text{on }\Omega\times(0,T)
\end{align}
with data satisfying Condition \ref{conditions} (C1)-(C5) and $R_1,\dots,R_\Phi\in L^2(0,T;H^1_D(\Omega)')$.
Additionally, we assume that the solution $u_0,\dots,u_n,\Phi$ satisfies Assumption \ref{assumptions} and Assumption \ref{ass:perturbed}.
In the case $z=0$, we do not need Assumption \ref{assumptions}.
In this paper, the solution to the perturbed system is the reconstruction $\widehat{u_0},\dots,\widehat{u_n},\widehat{\Phi}$ given in Section \ref{section:scheme} which satisfies Assumption \ref{ass:perturbed} automatically.
Using the stability framework, Theorem \ref{thm:abstract_general}, for the general model, we only need to bound the residuals $R_0,\dots,R_n,R_\Phi$ in the $L^2(0,T;H^1_D(\Omega)')$ norm to obtain a reliable a posteriori error estimate for the model \eqref{eqn:general}.
For the reduced model we combine the estimates for the solvent concentration from Section \ref{section:heat} and the bounds for the residuals $R_0,\dots,R_n,R_\Phi$ in the $L^2(0,T;H^1_D(\Omega)')$ norm to get a reliable a posteriori error estimate via the stability framework for the reduced model (see Theorem \ref{thm:abstract_error}).
This is done in Section \ref{section:residual} and relies mostly on the properties of the reconstruction.\\
We expect that the error $\widehat{u_i}-v_i$ converges linearly in the $L^2(0,T;H^1(\Omega))$ norm, since $\widehat{u_i}$ is manly composed of piecewise linear polynomials (see Section \ref{subsec:Morley}).
This linear convergence of the a posteriori estimator and error is observed in numerical experiments in Section \ref{section:numerics}.\\
The paper is structured as follows: In Section \ref{section:ineq}, we introduce the simplicical mesh and some inequalities used throughout this paper.
Next in Section \ref{section:scheme}, we present a cell centered finite volume scheme for the full system and a reconstruction of the finite volume solution similar to \cite{nicaise2005,nicaise2006}.\\
In Section \ref{section:heat}, we prove the bounds on $u_0-v_0$ and $\nabla(u_0-v_0)$ needed in the error estimator for the reduced model presented in Section \ref{section:abstract_error}.
After that, we present the stability framework for the reduced model \eqref{eqn:reduced} and the general model under the assumptions in Assumptions \ref{assumptions}.
The bounds for the residuals $R_0,\dots,R_n,R_\Phi$ in the $L^2(0,T;H^1_D(\Omega)')$-norm and the a posteriori error estimator for the general and reduced system are stated and proven in Section \ref{section:residual}.
In Section \ref{section:numerics}, the a posteriori error estimators are tested in numerical experiments.

\section{Notation and basic inequalities}\label{section:ineq}
\subsection{Notation}
We use the classical definition of a regular admissible finite volume mesh of $\Omega$ found in \cite[Definition 9.1]{eymard2019}.
The definition stated only allows for simplices, but the scheme in Section \ref{subsec:FV} can be used on any convex polygonal control volumes for every $d\geq 1$.
The reconstruction presented in Section \ref{subsec:Morley}, can be done in $d=2$ and $d=3$.
\begin{defn}{\cite[Def. 9.1]{eymard2019}}\label{defn:mesh}
	Let $\mathcal{T}$ be a set of simplices, $\mathcal{E}$ a family of edges with each $\sigma\in\mathcal{E}$ beeing a subset of $\overline{\Omega}$ and contained in a hyperplane of $\R^d$ and $\mathcal{P}$ a set of points in $\Omega$.
	The finite volume mesh $(\mathcal{T},\mathcal{E},\mathcal{P})$ is admissible if the following conditions are satisfied.
	\begin{enumerate}[(i)]
		\item $\bigcup_{K\in\mathcal{T}} \overline{K} = \overline{\Omega}$.
		\item For every $K,L\in\mathcal{T}$ with $K\neq L$ the $(d-1)$-dimensional Lebesgue measure of $\overline{K}\cap\overline{L}$ is 0 or there is $\sigma\in\mathcal{E}$ such that $\overline{\sigma}=\overline{K}\cap\overline{L}$.
		\item For every $K\in\mathcal{T}$, there is a subset $\mathcal{E}_K\subset\mathcal{E}$ such that $\partial K = \bigcup_{\sigma\in\mathcal{E}_K} \overline{\sigma}$ and $\mathcal{E} = \bigcup_{K\in\mathcal{T}}\mathcal{E}_K$.
		\item The set of points $\mathcal{P}$ has the form $\mathcal{P} = (x_K)_{K\in\mathcal{T}}$ with $x_K\in \overline{K}$ and $x_K\neq x_L$ for cells $K\neq L$. Furthermore, for cells $K,L\in\mathcal{T}$ that share an edge $\overline{\sigma} = \overline{K}\cap\overline{L}$ the straight line from $x_K$ to $x_L$ is orthogonal to $\sigma$.
		\item Let $\sigma\in\mathcal{E}$ with $\sigma\subset\partial \Omega$.
			Then $\sigma\subset\Gamma_N$ or $\sigma\subset\Gamma_D$.
			We further divide the set of edges into
			\begin{align*}
				\mathcal{E} = \mathcal{E}^i\cup\mathcal{E}^N\cup\mathcal{E}^D,
			\end{align*}
			where $\mathcal{E}^i$ are the inner edges, $\mathcal{E}^N$ are the edges on the Neumann boundary and $\mathcal{E}^D$ are the edges on the Dirichlet boundary.
	\end{enumerate}
\end{defn}
We denote by $n_{K,\sigma}$ the unit outer normal vector of $K$ on $\sigma\in\mathcal{E}_K$ and by $\{a_i^K\,:\,i=0,\dots,d\}\subset\overline{K}$ the corners of $K$.
Furthermore, we denote $N_\partial := \max_{K\in\mathcal{T}}\abs{\mathcal{E}_K}$.\\
The jump of a sufficiently regular function $f$ across an edge $\sigma\in\mathcal{E}$ is given by
\begin{align*}
	\llbracket f\rrbracket(x) := \begin{cases}
									f|_K(x)\cdot n_{K,\sigma}+f|_L(x)\cdot n_{L,\sigma}&\text{ if }\sigma=\partial K\cap\partial L\\
									f|_K(x)\cdot n_{K,\sigma} &\text{ if } \sigma\subset\partial\Omega
								\end{cases}
\end{align*}
for $x\in \sigma$.
We also define the average on an edge by
\begin{align*}
	\{\!\!\{f\}\!\!\}(x) :=
	\begin{cases}
		\frac12\left(f|_K(x)+f|_L(x)\right) &\text{ if } \sigma=\partial K\cap\partial L\\
		f|_K(x) &\text{ if } \sigma \subset\partial\Omega
	\end{cases}
\end{align*}
for $x\in\sigma$.
The classical element and edge bubble functions for any element $K\in\mathcal{T}$ and edge $\sigma\in\mathcal{E}$ are denoted by $b_K$ and $b_\sigma$ (for a definition see \cite[Lemma 4.3]{bartels2016}).\\
For every $K\in\mathcal{T}$ we set $h_K:=\operatorname{diam}(K)$ as the diameter of $K$ and for $\sigma\in\mathcal{E}$
\begin{align*}
	d_\sigma = 
	\begin{cases}
		\operatorname{dist}(x_K,x_L) &\text{if } \sigma=K\cap L\\
		\operatorname{dist}(x_K,\sigma)&\text{if } \sigma\subset \partial\Omega.	
	\end{cases}
\end{align*}
For each $K\in\mathcal{T}$ we define $\M_K:L^2(K)\to\R$ via
\begin{align*}
	\mathcal{M}_Kf = \frac1{\abs{K}}\int_K f(x)\;dx\quad\forall K\in\mathcal{T}, f\in L^2(K),\quad
\end{align*}
and for each $\sigma\in\mathcal{E}$ we define $M_\sigma:L^2(\sigma)\to\R$ via
\begin{align*}
	\mathcal{M}_\sigma f = \frac1{\abs{\sigma}}\int_\sigma f(x)\;dS(x)\quad\forall \sigma\in\mathcal{E}, f\in L^2(\sigma).
\end{align*}
For brevity, we write
\begin{align*}
	\norm{\cdot}_*:=\norm{\cdot}_{L^2(0,T;H^1_D(\Omega)')}, \quad
	\norm{\cdot}_M:=\norm{\cdot}_{L^2(M)},\quad
	X(q) := L^{\frac{2q}{q-d}}(0,T;L^q(\Omega)),
\end{align*}
for some measurable set $M$. 
For the time discretization, we fix a time $T$ and a partition $(t_j)_{j=0}^J$ with $t_0=0$ and $t_J = T$.
We denote the time step size by
\begin{align*}
	\tau_j = t_{j+1}-t_j\quad j=0,\dots,J-1.
\end{align*}
\subsection{Basic inequalities}
We need the following generalized version of the classical Poincar\'e-Friedrichs inequality.
A proof can be found in \cite[Theorem 3.1]{egert2015} and \cite[Section II.1.4]{teman1988}
\begin{thm}\label{thm:friedirchsineq}
	Let $\Omega\subset\R^d$ be an open domain and $\Gamma_1\subset\partial\Omega$ with $|\Gamma_1| >0$ and $1\leq q<\infty$.
	Then there is a constant $C_{F,q,\Gamma_1}>0$
	\begin{align*}
		\norm{u}_{L^q(\Omega)} \leq C_{F,q,\Gamma_1}\norm{\nabla u}_{L^q(\Omega)}
	\end{align*}
	for all $u\in H^1(\Omega)$ with $u|_{\Gamma_1}=0$.
\end{thm}
We later need a so called multiplicative trace theorem on each triangle.
This inequality can be found in \cite[Lemma 12.15]{ern2021}.
\begin{lemma}\label{scaledTrace}
	There exits $C_\text{cti}>0$ such that for all $p\in[1,\infty]$, $K\in \mathcal{T}$, $\sigma\in\mathcal{E}_K$ and $v\in W^{1,p}(K)$ it holds
	\begin{align*}
				\norm{v}_{L^p(\sigma)} 
		&\leq C_\text{cti}\left(\norm{\nabla v}_{L^p(K)}^{\frac1p}+h_K^{-\frac1p}\norm{v}_{L^p(K)}^{\frac1p}\right)\norm{v}_{L^p(K)}^{1-\frac1p}.
	\end{align*}
\end{lemma}
Combining Lemma \ref{scaledTrace} with the classical Poincar\'e inequality (see \cite[Corollary 2.3]{bartels2016}) we prove an approximation inequality on edges.
We denote the Poincar\'e constant for the space $L^q$ by $C_{P,q}$.

\begin{thm}\label{pFaceInterpolation}
	There exists $C_{\text{app}}>0$ such that for all $p\in[1,\infty]$, $K\in\mathcal{T}$, $\sigma\in\mathcal{E}_K$ and $v\in W^{1,p}(K)$ the following holds
	\begin{align*}
		\norm{v-\M_Kv}_{L^p(\sigma)} \leq C_{\text{app},p}h_K^{1-\frac1p}\norm{\nabla v}_{L^p(K)}.
	\end{align*}
\end{thm}
\begin{proof}
	Using Lemma \ref{scaledTrace} and Poincar\'e inequality we can bound 
	\begin{align*}
		\norm{v-\M_Kv}_{L^p(\sigma)}
		&\leq C_\text{cti}\left(\norm{\nabla v}_{L^p(K)}^{\frac1p}+h_K^{-\frac1p}\norm{v-\M_Kv}_{L^p(K)}^{\frac1p}\right)\norm{v-\M_Kv}_{L^p(K)}^{1-\frac1p}\\
		&\leq C_\text{cti}\left(\norm{\nabla v}_{L^p(K)}^{\frac1p}+C_{P,p}^{\frac1p}\norm{\nabla v}_{L^p(K)}^{\frac1p}\right)C_{P,p}^{1-\frac1p}h_K^{1-\frac1p}\norm{\nabla v}_{L^p(K)}^{1-\frac1p}\\
		&\leq C_\text{cti}(C_{P,p}^{1-\frac1p}+C_{P,p})h_K^{1-\frac1p}\norm{\nabla v}_{L^p(K)}
	\end{align*}
\end{proof}

In the following, we need a version of the Gagliardo-Nirenberg inequality for bounded domains.
A more general version was first proved in \cite{nirenberg1959} and the stated version can be found in \cite[Theorem 5.8]{adams2003}.
In \cite{adams2003}, the domain $\Omega$ needs to satisfy the cone condition.
Since we assume here that $\Omega$ has a locally Lipschitz boundary, the cone condition is satisfied for $\Omega$.
We use the Gagliardo-Nirenberg inequality for the stability framework in Section \ref{section:abstract_error} and present a simple proof such that we get an explicit bound on the constant $C_{G,p}$ that only depends on the constant $C_{S,q}$ in the Sobolev inequality for $q=\frac{2(d+p)}{d}$ for $d=2$ and $q=\frac{2d}{d-2}$ for $d>2$.
\begin{thm}\label{thm:Nirenberg}
	Let $2\leq p<\infty$ for $d\leq 2$ and $1\leq p<\frac{2d}{d-2}$ for $d>2$.
	There is a constant $C_{G,p}>0$ such that
	\begin{align*}
		\norm{u}_{L^p(\Omega)}\leq C_{G,p}\norm{u}_{H^1(\Omega)}^\theta\norm{u}_{L^2(\Omega)}^{1-\theta}\quad\forall u\in H^1(\Omega).
	\end{align*}
	with $\theta = \frac d2-\frac dp$.
	Furthermore, it holds $C_{G,p}\leq C_{S,q}^\theta$ with $q=\frac{2(d+p)}{d}$.
\end{thm}
\begin{proof}
	Since $u\in H^1(\Omega)$, the Sobolev inequality, see \cite[Theorem 4.12]{adams2003}, implies that $u\in L^q(\Omega)$.
	The $L^p$ interpolation inequality, see \cite[Theorem 2.11]{adams2003}, implies
	\begin{align*}
		\norm{u}_{L^p(\Omega)}\leq \norm{u}_{L^2(\Omega)}^{1-\theta}\norm{u}^\theta_{L^q(\Omega)}
	\end{align*}
	with $q=\frac{2(d+p)}{d}$ for $d=2$ and $q=\frac{2d}{d-2}$ for $d>2$ and $\theta = \frac d2-\frac dp$.
	We apply the Sobolev inequality and obtain
	\begin{align*}
		\norm{u}_{L^p(\Omega)}\leq C_{S,q}^\theta\norm{u}_{H^1(\Omega)}^\theta\norm{u}_{L^2(\Omega)}^{1-\theta},
	\end{align*}
	where $C_{S,q}$ is the Sobolev imbedding constant.
\end{proof}
\section{Numerical scheme and reconstruction}\label{section:scheme}
We use the finite volume scheme from \cite{gerstenmayer2018} to solve \eqref{equations}-\eqref{eqn:initial} in $d=1,2,3$ space dimensions and present a way to reconstruct the numerical solution so that a conforming approximation is obtained.
We provide a posteriori error estimates by reconstructing the numerical solution, in a way similar to \cite{nicaise2005,nicaise2006}.
There are also other finite volume schemes in the literature namely \cite{cances2019} and \cite{cances2024,cances2023}.
Since our goal is to initiate a posteriori error analysis for cross diffusion systems we opted to use the simpler finite volume scheme from \cite{gerstenmayer2018}.
In order to derive an a posteriori error estimate for the scheme from \cite{cances2023} we can combine the stability analysis from Section \ref{section:abstract_error} and the reconstruction presented below.
However, the use of this reconstruction would lead to an error estimator that is first order convergent, i.e. suboptimal for that scheme.
Deriving a reconstruction that provides a second order error estimator for the scheme from \cite{cances2023} is an interesting task but beyond the scope of this work.\\
\textcite{gerstenmayerFiniteelement2018} describes a structure preserving finite element method for \eqref{equations}-\eqref{eqn:initial}.
It solves the equation in the entropy variables and therefore asks for higher regularity conditions than the finite volume method presented below and all initial densities have to be positive everywhere.
\subsection{Finite Volume scheme}\label{subsec:FV}
In the following, we write $u_{i,K}^j$ for the approximation of $u_i$ on $K$ at time $t_j$.
We approximate the initial conditions as follows
\begin{align*}
	u_{i,K}^0 &= \M_K u_i^0,\quad\Phi_K^0 = \M_K\Phi^0
\end{align*}
for $i=0,\dots,n$ and $K\in\mathcal{T}$.
We recall the numerical scheme from \cite{gerstenmayer2018}. 
For $K\in\mathcal{T}$ and $j\in\{1,\dots,J\}$
\begin{align}\label{eqn:scheme}
	\abs{K}\frac{u_{i,K}^j-u_{i,K}^{j-1}}{\tau_j} + \sum_{\sigma\in\mathcal{E}_K} \mathcal{F}_{i,K,\sigma}^j &= 0\quad \forall i\in\{1,\dots,n\}\\
	u_{0,K}^j &= 1-\sum_{i=1}^nu_{i,K}^j\\
	-\lambda^2\sum_{\sigma\in\mathcal{E}_K} F_\sigma^K(\Phi^j) - \abs{K}\sum_{i=1}^n z_iu_{i,K}^j &= 0
	\label{eqn:scheme_phi}
\end{align}
with
\begin{align}
	\mathcal{F}^j_{i,K,\sigma} := -u_{0,\sigma}^jF_{\sigma}^K(u_i^j)+u_{i,\sigma}^j (F_{\sigma}^K(u_{0}^j)-u_{0,\sigma}^j\beta z_iF_{\sigma}^K(\Phi^j))\quad i=1,\dots,n
\end{align}
and
\begin{align*}
	F^K_{\sigma}(u_{i}^j) &:= \begin{cases}
		\frac{\abs{\sigma}}{d_\sigma}(u_{i,K}^j-u_{i,L}^j) &\text{if } \sigma = K\cap L\\
		0 &\text{if } \sigma \subset \Gamma_N\\
		\frac{\abs{\sigma}}{d_\sigma}u_{i,K}^j &\text{if } \sigma \subset \Gamma_D
	\end{cases},\quad
	u_{i,\sigma}^j
	:= \begin{cases}
		u_{i,K}^j &\text{if } \sigma\subset\Gamma_N\\
		\M_\sigma u_i^D &\text{if } \sigma\subset\Gamma_D\\
		\frac{u_{i,K}^j-u_{i,L}^j}{\log(u_{i,K}^j)-\log(u_{i,L}^j)}&\text{otherwise}
	\end{cases}\quad i=0,\dots,n.
\end{align*}
For the reconstruction of the solvent concentration we need a numerical equivalent of equation \eqref{eqn:solvent}.
Indeed, summing \eqref{eqn:scheme} over $i$ yields
\begin{align}\label{eqn:scheme_0}
	\abs{K}\frac{u_{0,K}^j-u_{0,K}^{j-1}}{\tau_j}-\sum_{\sigma\in\mathcal{E}_K}F^K_{\sigma}(u_0^j)\left(u_{0,\sigma}^j+\sum_{i=1}^n u_{i,\sigma}^j\right)+\beta F^K_{\sigma}(\Phi^j)u_{0,\sigma}^j\left(\sum_{i=1}^n z_iu_{i,\sigma}^j\right)=0.
\end{align}
We write $u_{i,h}$ for the piecewise constant functions on $\Omega$ given by $u_{i,h}(x) = u_{i,K}$ for all $x\in K$ and $i=0,\dots,n$.
Similarly, we define $\Phi_h$ by $\Phi_h(x)=\Phi_K$ for all $x\in K$.
\begin{remark}\label{remark:scheme}
	\begin{enumerate}[(i)]
		\item If the mesh $\mathcal{T}$ does not satisfy the orthogonality criterion, one can use the flux defined in \cite{coudiere1999} and get a similar scheme.
			What is crucial for our reconstruction to work is conservation of mass, i.e. $F_{\sigma,K}^j(f) = -F_{\sigma,L}^j(f)$ for $\sigma=\partial K\cap\partial L$, and this is also satisfied for the fluxes used in \cite{coudiere1999,nicaise2005,nicaise2006}.
		\item We use the logarithmic-mean to get an approximation of $u_i$ on the edges.
			\textcite{cances2019} use an upwind-scheme and \textcite{cances2023} use the arithmetic mean if $z_i=0$ ($i=1,\dots,n$).
			Using the arithmetic mean leads to similar results as using the logarithmic-mean in our tests.
			The same holds for the use of other means, e.g. the geometric mean.
	\end{enumerate}
\end{remark}
\subsection{Morley type reconstruction}\label{subsec:Morley}
We reconstruct the finite volume solution $(u_{0,h},\dots,u_{h,n},\Phi_h)$ defined in \eqref{eqn:scheme}-\eqref{eqn:scheme_phi}.
We present here only the case $d=2$.
The case $d=3$ works similarly on tetrahedral meshes.
For the spacial reconstruction we use a Morley type reconstruction similar to \cite{nicaise2005,nicaise2006}.
For every time step, we first interpolate the finite volume solution with continuous piecewise linear functions.
After that the integral of the outward normal derivative of the interpolant is adjusted on every edge to fit the numerical diffusive flux.\\
In \cite{nicaise2005} similar reconstructions are applied to the Poisson equation and in \cite{nicaise2006} to more general elliptic equations.
The main difference to \cite{nicaise2006} is that we do not fix averages along edges in the reconstruction process.
First, we reconstruct the solvent concentration $u_{0,h}$ and the electric potential $\Phi_h$, since they solve equations independent of the ion species.
We use the following polynomial space on $K\in\mathcal{T}$
\begin{align*}
	P_K &= \{q+pb_K\;:\;q\in\mathbb{P}_1(K),p\in\mathbb{P}_1(K)\}.
\end{align*}
For the solvent concentration and electric potential we use the degrees of freedom
\begin{align*}
	\Sigma_{K,0} &= \{v(a_i^K)\}_{i=1,2,3}\cup\left\{\int_\sigma\frac{\partial v}{\partial n_{K,\sigma}}dS\right\}_{\sigma\in\mathcal{E}_K}
\end{align*}
on every element $K\in\mathcal{T}$.
We define the Morley finite element space for the reconstructions of the $u_0^j$ and $\Phi^j$ by
\begin{align*}
	V_h^0 = \left\{v_h\in H^1_D(\Omega)\;:\;\right.&v_h|_K\in P_K\;\forall K\in\mathcal{T},\\
										&\int_\sigma \frac{\partial v_h|_K}{\partial n_{K,\sigma}} dS = -\int_\sigma \frac{\partial v_h|_L}{\partial n_{L,\sigma}} dS \;\forall \sigma\in\mathcal{E},K,L\in\mathcal{T}:\sigma = K\cap L,\\
										&\left.\int_\sigma \frac{\partial v_h|_K}{\partial n_{K,\sigma}} dS = 0 \;\forall K\in\mathcal{T}:\sigma\in\mathcal{E}^N\cap\mathcal{E}_K\right\}
\end{align*}
With the above notation on hand we can define the Morley type reconstruction for the solvent concentration and electric potential.
\begin{defn}\label{defn:Morley_solvent}
	Let the weights $w_L(a_i^K)$ ($i=1,\dots,3$ and $K,L\in\mathcal{T}$ with $a_i^K\in\overline{K}\cap\overline{L}$) be given according to \cite[Section 3.3]{coudiere1999}.
	\begin{enumerate}[(i)]
		\item For a finite volume solution $u_0^j$ of \eqref{eqn:scheme}-\eqref{eqn:scheme_phi} we define its Morley type reconstruction $\widehat{u_0}^j$ as the unique element $\widehat{u_0}^j\in V_h^0$ satisfying
			\begin{align*}
				\widehat{u_0}^j(a_i^K) &= \sum_{L\in\mathcal{T}:a_i^K\in \overline{L}}w_L(a_i^K)u_0^j(L)\quad\forall K\in\mathcal{T},i\in\{1,2,3\}\\
				\int_\sigma \frac{\partial \widehat{u_0}^j|_K}{\partial n_{K,\sigma}}\, dS &= F_{\sigma}^K(u_{0}^j)\left(u_{0,\sigma}^j+\sum_{i=1}^n u_{i,\sigma}^j\right)\quad\forall K\in\mathcal{T},\sigma\in \mathcal{E}_K.
			\end{align*}
		\item For a finite volume solution $\Phi^j$ of \eqref{eqn:scheme}-\eqref{eqn:scheme_phi} we define its Morley type reconstruction $\widehat{\Phi}^j$ as the unique element $\widehat{\Phi}^j\in V_h^0$ satisfying
			\begin{align*}
				\widehat{\Phi}^j(a_i^K) &= \sum_{L\in\mathcal{T}:a_i^K\in \overline{L}}w_L(a_i^K)\Phi^j(L)\quad\forall K\in\mathcal{T},i\in\{1,2,3\}\\
				\int_\sigma \frac{\partial \widehat{\Phi}^j|_K}{\partial n_{K,\sigma}}\, dS &= F_{\sigma}^K(\Phi^j)\quad\forall K\in\mathcal{T},\sigma\in \mathcal{E}_K.
			\end{align*}
	\end{enumerate}
\end{defn}
\begin{remark}
	\begin{enumerate}[(i)]
		\item
			The Morley type reconstruction for the solvent concentration is different from that presented in \cite{nicaise2005} and \cite{nicaise2006}.
			In \cite{nicaise2005} the $C^0$ reconstruction on triangles uses the polynomial base $q+pb_K$ for $q\in\mathbb{P}_2(K)$ instead of $\mathbb{P}_1(K)$ and use $\{p(m_i)\}_{i=1,2,3}$ (with $m_i$ the edge-midpoints of the triangle $K$) in addition to the degrees of freedom used here.
			For simplicity, we stick to the reconstruction presented in Definition \ref{defn:Morley_solvent}, since it uses less degrees of freedom and fulfills our needs.
		\item In the definition of $\widehat{u_0}^j$ we used $F_\sigma^K(u_0^j)\left(u_{0,\sigma}^j+\sum_{i=1}^nu_{i,\sigma}^j\right)$ instead of $F_\sigma^K(u_0^j)$ for the integral of the normal derivatives.
			In the proof of Theorem \ref{thm:H1_heat_bound} and Proposition \ref{prop:MorleyOrth}, we need that the integral of the normal derivatives is given by the numerical fluxes used in the finite volume scheme to derive the orthogonality condition \eqref{eqn:orth}.
			Recall that $\left(u_{0,\sigma}^j+\sum_{i=1}^nu_{i,\sigma}^j\right)\approx 1$.
	\end{enumerate}
\end{remark}
For the other species $u_i$ ($i=1,\dots,n$) we also use the polynomial space $P_K$, but with degrees of freedom
\begin{align*}
	\Sigma_K^j &= \{v(a_i^K)\}_{i=1,2,3}\cup\left\{\int_\sigma \widehat{u_0}^j\frac{\partial v}{\partial n_{K,\sigma}}dS\right\}_{\sigma\in\mathcal{E}_\sigma}
\end{align*}
on every triangle and time step $j=0,\dots,J$.
We now prove that $(K,P_K,\Sigma_K^j)$ defines a finite element for every $K\in\mathcal{T}$.
The proof is similar to \cite[Lemma 4.1]{nicaise2006} and \cite[Lemma 4.1]{nilssen2000}.
\begin{prop}
	If $\widehat{u_0}^j>0$ on $\overline{K}$, then the triple $(K,P_K,\Sigma_K^j)$ is a $C^0$-finite element.
\end{prop}
\begin{proof}
	Since $\dim P_K =6 =\operatorname{card}\Sigma_K$ it suffices to show that if $w=q+pb_K\in P_K$
	\begin{align*}
		\lambda(w) = 0\quad\forall\lambda\in\Sigma_K^j
	\end{align*}
	implies $w=0$.
	The condition $0=w(a_i^K) = q(a_i^K)+p(a_i^K)b_K(a_i^K) = q(a_i^K)$ and $q\in \mathbb{P}_1(K)$ implies $q=0$.\\
	Let $\sigma\in\mathcal{E}_K$ be an arbritrary edge of $K$.
	Since the zero-set of $b_\sigma$ is a subset of the zero-set of $b_K$ we can factor out $b_\sigma$
	\begin{align*}
		b_K = b_\sigma \lambda_\sigma,
	\end{align*}
	with $\lambda_\sigma \neq 0$ in the interior of $\sigma$.
	The properties of $b_\sigma$ imply that $\frac{\partial b_\sigma}{\partial n_{\sigma,K}}<0$ on the interior of $\sigma$.
	For every edge $\sigma\in\mathcal{E}_K$ holds
	\begin{align*}
		0=\int_\sigma \widehat{u_0}^j\frac{\partial w}{\partial n_{\sigma,K}}\,dS
		=\int_\sigma \widehat{u_0}^j\frac{\partial (pb_K)}{\partial n_{\sigma,K}}\,dS
		=\int_\sigma \widehat{u_0}^jp\lambda_\sigma\frac{\partial b_\sigma}{\partial n_{\sigma,K}}\,dS.
	\end{align*}
	Thus, $\widehat{u_0}^jp$ has a root on the interior of every edge $\sigma\in\mathcal{E}_K$.
	Since $\widehat{u_0}^j>0$, we conclude, that $p\in\mathbb{P}_1(K)$ has a root on the interior of every edge $\sigma\in\mathcal{E}_K$.
	This implies $p=0$ and with the above $w=0$.
\end{proof}
We define the Morley finite element space for the reconstructions of $u_{i,h}^j$ by
\begin{align*}
	V_{j,h} = \left\{v_h\in H^1_D(\Omega):\right.&v_h|_K\in P_K\;\forall K\in\mathcal{T}\\
										&\int_\sigma \widehat{u_0}^j\frac{\partial v_h|_K}{\partial n_\sigma} dS = \int_\sigma \widehat{u_0}^j\frac{\partial v_h|_L}{\partial n_\sigma} dS \;\forall \sigma\in\mathcal{E},K,L\in\mathcal{T}:\sigma = K\cap L,\\
										&\left.\int_\sigma \widehat{u_0}^j\frac{\partial v_h|_K}{\partial n_\sigma} dS = 0 \;\forall \sigma\in\mathcal{E}^N\cap\mathcal{E}_K\right\}.
\end{align*}
\begin{defn}
	The Morley type reconstruction $\widehat{u_i}^j$ for the species $u_i^j$ at time $j$ is the unique element $\widehat{u_i}^j\in V_{j,h}$ satisfying
	\begin{align*}
		\widehat{u_i}^j(a_i^K) &= \sum_{L\in\mathcal{T}:a_i^K\in \overline{L}}w_L(a_i^K)u_{i,L}^j\quad\forall K\in\mathcal{T},i\in\{1,2,3\}\\
		\int_\sigma \widehat{u_0}^j\frac{\partial \widehat{u_i}^j|_K}{\partial n_{K,\sigma}}\, dS &= u_{0,\sigma}^jF^K_\sigma(u_{i}^j)\quad\forall K\in\mathcal{T},\sigma\in \mathcal{E}_K,
	\end{align*}
	where we again use the weights $w_L(a_i^K)$ chosen according to \cite{coudiere1999}
\end{defn}
\begin{remark}\label{remark:convection}
	\textcite{nicaise2006} uses terms of the form $\alpha_\sigma b_\sigma b_K$ ($\alpha_\sigma\in\R$ for every edge $\sigma\in\mathcal{E}_K$) to add degrees of freedom corresponding to the convective fluxes.
	In our experiments, the convection term $u_{i,\sigma}^jF_\sigma^K(u_0^j)$ is large compared to $u_i^j$ so that the coefficient $\alpha_\sigma$ becomes large and introduces artificial bumps of the form $b_\sigma b_K$ in the middle of the triangle $K$.
	This yields convergence rates of order $\frac12$ rather than linear convergence.
\end{remark}
For the reconstruction in time we use a linear interpolation between the time steps.
Namely, we denote the space-time reconstruction by
\begin{align*}
	\widehat{u_i}(t) := \frac{t-t_{j-1}}{\tau_j} \widehat{u_i}^j + \frac{t_j-t}{\tau_j} \widehat{u_i}^{j-1}\quad\text{ for }i=0,\dots,n,
\end{align*}
for $t\in[t_{j-1},t_j]$.
We use a linear reconstruction in time, since it is easy and is sufficient for our needs.
\section{Error estimates for the diffusion equation}\label{section:heat}
The solvent concentration in the reduced model \eqref{eqn:reduced} solves the diffusion equation.
In Theorem \ref{thm:abstract_error}, we derive a bound of the difference $u_i-v_i$ that requires bounds for $\norm{u_0-v_0}_{L^\infty([0,T]\times\Omega)}$ and $\norm{\nabla (u_0-v_0)}_{L^2([0,T]\times\Omega)}$.
In this section, we derive such bounds if $u_0$ is given by the reconstruction of the finite volume solution given in Section \ref{section:scheme}.
Following a similar strategy as in \cite{kyza2018} and \cite{demlow2023}, we first derive the $L^\infty$ bound with the help of Green's function for the Poisson equation on $\Omega$.
The proof presented below uses a uniform with respect to the second argument bound on the $W^{1,p}$ norm of Green's function on $\Omega$ for some $p^*\in\left[1,\frac{d}{d-1}\right)$ (see Condition \ref{conditions} (C6)).
For the $L^2$ bound of the gradient of $u_0-v_0$ we use the stability estimate from \cite[Proposition 4.5]{bartels2016}.
The residual is then estimated by the properties of the reconstruction similar to \cite{nicaise2005,nicaise2006}.\\
\subsection{$L^\infty$ norm estimates for the diffusion equation}
We look at the classical diffusion equation with Dirichlet and Neumann boundary conditions
\begin{align}\label{eqn:heat}
	\begin{split}
		\partial_t v&= \Delta v\quad \text{on }\Omega\times[0,T]\\
		v &= v^D \quad \text{on }\Gamma_D\times[0,T]\\
		\nabla v\cdot n &= 0\quad\text{on }\Gamma_N\times[0,T]
	\end{split}
\end{align}
with $\Omega,\Gamma_D$ and $\Gamma_N$ as in Conditions \ref{conditions} (C1),(C7) and initial and boundary conditions satisfying Conditions (C2),(C3),(C6).
More on these conditions in Remark \ref{remark:uniform}.
To show an upper error bound in the maximum norm for the reconstruction of the finite volume solution we follow the approach of \cite{kyza2018, demlow2012, demlow2023}.
In the upcoming discussion we need Green's function $G$ for the domain $\Omega$ satisfying
\begin{align*}
	-\Delta_x G(x;y) = \delta(y-x)\quad \forall x,y\in\Omega,
\end{align*}
with boundary conditions
\begin{align*}
	\nabla G\cdot n &= 0\quad\text{on }\Gamma_N\\
	G &=0\quad\text{on }\Gamma_D.\\
\end{align*}
This yields using Green's identity
\begin{align}\label{eqn:Green}
	f(y) = \int_\Omega \nabla G(x;y)\cdot\nabla f(x)\;dx\quad\forall f\in H^1_D(\Omega)\cap W^{1,q}(\Omega)
\end{align}
with $q> d$ which implies that $W^{1,q}(\Omega)$ embedds into $C^0(\Omega)$.
In the following, we use an elliptic reconstruction in every time step $t_j$.
We define the elliptic reconstruction $w^j$ for time step $j$ as the solution of
\begin{align}\label{def_ell_recon}
	\int_\Omega \nabla w^j\cdot\nabla f\,dx &= \int_\Omega A^jf\,dx\quad\forall f\in H^1_D(\Omega)
\end{align}
with
\begin{align}\label{eqn:def_A^j}
	A^j := \begin{cases}
		-\Delta_h u_{0,h}^0 &\text{if } j=0\\
			-\frac{u_0^j-u_0^{j-1}}{\tau_j} &\text{if } j>0.
		\end{cases}
\end{align}
and boundary conditions
\begin{align*}
	\nabla w^j\cdot n &= 0\quad\text{on }\Gamma_N\\
	w^j &= u^D\quad\text{on }\Gamma_D.
\end{align*}
In \eqref{eqn:def_A^j}, we use the discrete Laplacian defined by
\begin{align*}
	-\Delta_h u_{0,h}^0|_K =\frac1{\abs{K}}\sum_{\sigma\in\mathcal{E}_K}\left(u_{0,\sigma}^j+\sum_{i=1}^nu_{i,\sigma}^j\right) F_\sigma^K(u_0^0)\quad\forall K\in\mathcal{T}.
\end{align*}
The reason for this is, that we cannot use $\Delta \widehat{u_0}^0$ for $w^0$, since $\Delta\widehat{u_0}^0\notin L^2(\Omega)$ and we need $-A^0\in L^2(\Omega)$ in Lemma \ref{lemma:ParabolicBound}.
Condition \ref{conditions} (C6) and (C7) then guarantee that the elliptic reconstruction $w^0\in H^2(\Omega)$ (see Remark \ref{remark:uniform}.\\
Following \cite{demlow2009}, we split the error $e:=v_0-\widehat{u_0}$ into an elliptic error $\varepsilon := w-\widehat{u_0}$ and a parabolic error $\rho:=v_0-w$.
Due to \eqref{eqn:heat}, $\widehat{u_0}^j$ can be understood as a numerical approximation of the Poisson equation $-\Delta w = -A^j$.
Thus, we can use error estimates for elliptic problems to get a pointwise error bound for $\varepsilon$.
The parabolic error $\rho$ solves the diffusion equation
\begin{align}\label{eqn:parabolic_error}
	\partial_t \rho -\Delta \rho = R(t)-\partial_t\varepsilon
\end{align}
with the temporal residual $R(t):=-\ell_{j-1}A^{j-1}-\ell_jA^j-\frac{\widehat{u_0}^j-\widehat{u_0}^{j-1}}{\tau_j}$ for $t\in[t_{j-1},t_j]$.
Therefore, we can use Duhamel's principle and the heat semigroup to bound $\rho$.\\
We will proceed as follows: First, we bound $\varepsilon$ pointwise using the properties of the Morley type reconstruction and Green's function at each time level $t_j$.
We can represent the elliptic error at time $t_j$ by \eqref{eqn:Green} and obtain a pointwise bound by estimating the right hand side.
After that, we use the $L^\infty$-contractivity of the heat semigroup on $\Omega$ with homogeneous Neumann and Dirichlet boundary conditions to bound $\norm{\rho}_{L^\infty([0,T]\times\Omega)}$.
For this we also use the elliptic error estimate that was derived in the first step.
Combining the error estimates for $\norm{\varepsilon^j}_{L^\infty([0,T]\times\Omega)}$ and $\norm{\rho}_{L^\infty([0,T]\times\Omega)}$, we obtain the error estimate for $\norm{e}_{L^\infty([0,T]\times\Omega)}$.
\begin{remark}\label{remark:uniform}
			Condition \ref{conditions} (C6) and (C7) guarantees that the elliptic reconstruction satisfies $w^j\in W^{1,q}(\Omega)$ with $\frac1q+\frac1{p^*}=1$.
			Indeed, the elliptic reconstruction $w^j$ can be decomposed $w^j=\tilde{w}^j+u^D$ with $\tilde{w}^j\in H^1_D(\Omega)$.
			Using \eqref{def_ell_recon} we can write
			\begin{align*}
				w^j(y) = \int_\Omega G(x;y) A^j(x)\,dx+u^D(y).
			\end{align*}
			Since $A^j\in L^\infty(\Omega)$, we get that $w^j\in W^{1,q}(\Omega)$.
\end{remark}
We first prove a bound of the elliptic error $\varepsilon$ for every time step.
The next proposition establishes a quasi-orthogonality property for $\varepsilon$ to piecewiese constant functions similar to \cite[Lemma 5.3]{nicaise2005}.
Since we cannot use piecewiese constant functions as test functions in the weak formulation of the Poisson problem, the error cannot be orthogonal to those functions.
The proposition below is central to the following considerations.
\begin{prop}\label{prop:MorleyOrth}
	Let $1\leq p\leq p^*$.
	The elliptic error $\varepsilon^j := w^j -\widehat{u_0}^j$ at time step $j$ satisfies
	\begin{multline*}
		\int_\Omega\nabla (w^j-\widehat{u_0}^j)\cdot\nabla f\,dx = \sum_{K\in \mathcal{T}}\left(\int_K (-A^j-\Delta \widehat{u_0}^j)(f-\M_K f) dx\right.\\
										\left.-\frac12\sum_{\sigma\in\mathcal{E}_K\setminus\mathcal{E}^D}\int_{\sigma} \jump{\nabla \widehat{u_0}^j}(f-\M_K f) dS\right) \quad \forall f\in W^{1,p}_D(\Omega)
	\end{multline*}
\end{prop}
\begin{proof}
	Let $f\in W^{1,p}_D(\Omega)$.
	Similar to \cite[Lemma 5.3]{nicaise2005}, we conclude using the definition of the reconstruction (see Definition \ref{defn:Morley_solvent})
	\begin{align}\label{eqn:same_mean}
		\int_K \Delta \widehat{u_0}^j\,dx
		=\sum_{\sigma\in\mathcal{E}_K}\int_\sigma\nabla \widehat{u_0}^j\cdot n_{K,\sigma}\,dS
		=\sum_{\sigma\in\mathcal{E}_K} \left(u_{0,\sigma}^j+\sum_{i=1}^nu_{i,\sigma}^j\right)F_\sigma^K(u_0^j)
		= \int_K -A^j\,dx
	\end{align}
	for every $K\in\mathcal{T}$.
	With \eqref{eqn:same_mean} and integration by parts we derive
	\begin{align}\label{eqn:orth}
		\int_\Omega \nabla (w^j-\widehat{u_0}^j)\cdot\nabla f\,dx
		&=\sum_{K\in \mathcal{T}}\int_K (-A^j-\Delta \widehat{u_0}^j)(f-\M_K f) dx -\frac12\sum_{K\in\mathcal{T}}\sum_{\sigma\in\mathcal{E}_K\setminus\mathcal{E}^D}\int_\sigma\left\llbracket\nabla \widehat{u_0}^j\right\rrbracket f ds
	\end{align}
	Due to the conservation property of the numerical flux, i.e.\\
	$\int_\sigma \nabla \widehat{u_0}^j\cdot n_{K,\sigma}\,dS = F_\sigma^K(u_0^j)=-F_\sigma^L(u_0^j)=-\int_\sigma \nabla \widehat{u_0}^j\cdot n_{L,\sigma}\,dS$ for $\sigma = \overline{K}\cap\overline{L}$, we get
	\begin{multline*}
		\int_\Omega\nabla (w^j-\widehat{u_0}^j)\cdot\nabla f\,dx=\sum_{K\in \mathcal{T}}\int_K (-A^j-\Delta \widehat{u_0}^j)(f-\M_K f)\, dx\\
		-\frac12\sum_{K\in\mathcal{T}}\sum_{\sigma\in\mathcal{E}_K\setminus\mathcal{E}^D}\int_\sigma\left\llbracket\nabla \widehat{u_0}^j\right\rrbracket (f-\M_K f) dS.
	\end{multline*}
\end{proof}
We can now estimate the elliptic error in the maximum norm for every time step.
For this we use Proposition \ref{prop:MorleyOrth} with $f=G$ and Theorem \ref{scaledTrace} to estimate the jump terms.
\begin{lemma}\label{lemma:EllipticBound}
	If Condition \ref{conditions} (C1)-(C3),(C6) and (C7) are satisfied with $p^*$, then the elliptic error $\varepsilon^j$ at time $t_j$ can be bounded by
	\begin{align*}
		\norm{\varepsilon^j}_{L^\infty}\leq \eta_{S,q}^jC_{\text{Green},p}
	\end{align*}
	with $\frac1{p^*}+\frac1q=1$ and
	\begin{align*}
		\eta_{S,q}^j &:=
		\begin{dcases}
			2^{\frac1{p^*}}\left(\sum_{K\in\mathcal{T}}\left(C_{P,p^*}^qh_K^q\norm{A^j+\Delta \widehat{u_0}^j}_{L^q(K)}^q+\frac{C_{\text{app},p^*}^qN_\partial^{q-1}}{2^q}h_K\sum_{\sigma\in\mathcal{E}_K\setminus\mathcal{E}^D}\norm{\llbracket\nabla \widehat{u_0}^j\rrbracket}_{L^q(\sigma)}^q\right)\right)^\frac1q &\text{for }q<\infty\\[2\baselineskip]
			\max_{K\in\mathcal{T}}\left(C_{P,1}h_K\norm{A^j+\Delta \widehat{u_0}^j}_{L^\infty(K)}+\frac{C_{\text{app},1}}{2}\sum_{\sigma\in\mathcal{E}_K\setminus\mathcal{E}^D}\norm{\llbracket\nabla \widehat{u_0}^j\rrbracket}_{L^\infty(\sigma)}\right)&\text{for }q=\infty.
		\end{dcases}
	\end{align*}
\end{lemma}
\begin{proof}
	Using \eqref{eqn:Green} and Proposition \ref{prop:MorleyOrth} yields
	\begin{align*}
		(w^j-\widehat{u_0}^j)(y)
		&=\int_\Omega \nabla (w^j-\widehat{u_0}^j)\cdot\nabla G(\cdot;y)\,dx\\
		&=\sum_{K\in\mathcal{T}}\int_K (-A^j-\Delta \widehat{u_0}^j)(G-\M_K G) dx
		-\frac12\sum_{\sigma\in\mathcal{E}_K\setminus\mathcal{E}^D}\int_{\sigma} \left\llbracket\nabla \widehat{u_0}^j \right\rrbracket(G-\M_K G)\, dS\\
		&\leq \sum_{K\in\mathcal{T}}\norm{-A^j-\Delta \widehat{u_0}^j}_{L^q(K)}\norm{G-\M_KG}_{L^{p^*}(K)}
		+\frac12\sum_{\sigma\in\mathcal{E}_K\setminus\mathcal{E}^D}\norm{\left\llbracket\nabla \widehat{u_0}^j\right\rrbracket}_{L^q(\sigma)}\norm{G-\M_K G}_{L^{p^*}(\sigma)}.
	\end{align*}
	Using Theorem \ref{pFaceInterpolation} and the Cauchy-Schwarz inequality we can estimate
	\begin{align*}
		\int_\Omega\nabla (w^j-\widehat{u_0}^j)\cdot\nabla G\,dx
		&\leq \sum_{K\in\mathcal{T}}\norm{-A^j-\Delta \widehat{u_0}^j}_{L^q(K)}C_{P,p^*}h_K\norm{\nabla G}_{L^{p^*}(K)}\\
		&+\frac12\sum_{\sigma\in\mathcal{E}_K\setminus\mathcal{E}^D}\norm{\left\llbracket\nabla \widehat{u_0}^j\right\rrbracket}_{L^q(\sigma)}C_{\text{app},p^*}h_K^{1-\frac1{p^*}}\norm{\nabla G}_{L^{p^*}(K)}\\
		&\leq2^{\frac1{p^*}}\left(\sum_{K\in\mathcal{T}}\left(C_{P,p^*}^qh_K^q\norm{A^j+\Delta \widehat{u_0}^j}_{L^q(K)}^q +\frac{C_{\text{app},p^*}^qN_\partial^{q-1}}{2^q} h_K\sum_{\sigma\in\mathcal{E}_K\setminus\mathcal{E}^D}\norm{\llbracket\nabla \widehat{u_0}^j\rrbracket}_{L^q(\partial K)}^q\right)\right)^{\frac1q}\\
		&\left(\sum_{K\in\mathcal{T}}\norm{\nabla G}_{L^{p^*}(K)}^{p^*}\right)^{\frac1{p^*}}\\
		&\leq \eta_{S,q}^j\norm{\nabla G}_{L^{p^*}(\Omega)}.
	\end{align*}
	Where $N_\partial=\max_{K\in\mathcal{T}}\abs{\mathcal{E}_K}$, see Definition \ref{defn:mesh}.
	The proof for $q=\infty$ is completely analogous.
\end{proof}
We now study the parabolic error $\rho$.
For this we need the heat semigroup on $\Omega$ with homogeneous Neumann and Dirichlet conditions on $\Gamma_N$ and $\Gamma_D$ respectively, denoted by $\left(e^{t\Delta}\right)_{t\geq0}$.
Furthermore, we note that in \cite[Chapter 4]{ouhabaz2009} it is proved that
\begin{align*}
	\norm{e^{t\Delta}f}_{L^\infty(\Omega)}\leq \norm{f}_{L^\infty(\Omega)}
\end{align*}
for every $f\in L^\infty(\Omega)$ and $t>0$.
Notice, that since $w$ and $v_0$ satisfy the boundary conditions exactly, we only need to consider homogeneous Dirichlet and Neumann boundary conditions.
Further, we define the temporal residual
\begin{align*}
	R(t,\cdot) = -\ell_{j-1}(t)A^{j-1}-\ell_{j}(t)A^j-\frac{\widehat{u_0}^j-\widehat{u_0}^{j-1}}{\tau}\quad t\in[t_{j-1},t_j]
\end{align*}
with $\ell_{j-1}(t) := \frac{t_j-t}{\tau_j}$ and $\ell_j(t):=\frac{t_{j-1}+t}{\tau_j}$.
\begin{lemma}\label{lemma:ParabolicBound}
	If Condition \ref{conditions} (C1)-(C3), (C6) and (C7) are satisfied, then the parabolic error $\rho := v_0-w$ satisfies for $t\in (t_{j-1},t_j)$ and $j=1,\dots,J$
	\begin{align*}
		\norm{\rho(t,\cdot)}_{L^\infty(\Omega)} \leq \norm{\rho(t_{j-1},\cdot)}_{L^\infty(\Omega)}+\eta_{T,\infty}^j+\tau_j\dot{\eta}_{S,q}^j C_{\text{Green},p^*}
	\end{align*}
	with $\frac1{p^*}+\frac1q=1$ and
	\begin{alignat*}{2}
		\eta_{T,\infty}^j &:= \int_{t_{j-1}}^{t_j}\norm{R(t,\cdot)}_{L^\infty(\Omega)}\;dt = \int_{t_{j-1}}^{t_j}\norm{-\ell_{j-1}A^{j-1}-\ell_jA^j-\frac{\widehat{u_0}^j-\widehat{u_0}^{j-1}}{\tau_j}}_{L^\infty(\Omega)} dt&&\\
		\dot{\eta}_{S,q}^j&:= \tau_j^{-1}2^{\frac1{p^*}}\left(\sum_{K\in\mathcal{T}}\left(C_{P,p^*}^qh_K^q\norm{A^j-A^{j-1}+\Delta (\widehat{u_0}^j-\widehat{u_0}^{j-1})}_{L^q(K)}^q\right.\right.\\
			&\quad\left.\left.+\frac{C_{\text{app},p^*}^qN_\partial^{q-1}}{2^q}h_K\sum_{\sigma\in\mathcal{E}_K\setminus\mathcal{E}^D}\norm{\left\llbracket\nabla (\widehat{u_0}^j-\widehat{u_0}^{j-1})\right\rrbracket}_{L^q(\sigma)}^q\right)\right)^{\frac1q}&& \text{for }q<\infty\\
		\dot{\eta}_{S,\infty}^j&:=\tau_j^{-1}\max_{K\in\mathcal{T}}\left(C_{P,1}h_K\norm{A^j-A^{j-1}+\Delta (\widehat{u_0}^j-\widehat{u_0}^{j-1})}_{L^\infty(K)}\right.\\
			&\quad\left.+\frac{C_{\text{app},1}}{2}\sum_{\sigma\in\mathcal{E}_K\setminus\mathcal{E}^D}\norm{\left\llbracket\nabla (\widehat{u_0}^j-\widehat{u_0}^{j-1})\right\rrbracket}_{L^\infty(\sigma)}\right) &&\text{for } q=\infty.
	\end{alignat*}
\end{lemma}
\begin{proof}
	One can easily verify that
	\begin{align*}
		\partial_t\rho-\Delta\rho = -\ell_{j-1}A^{j-1}-\ell_{j}A^j-\frac{\widehat{u_0}^j-\widehat{u_0}^{j-1}}{\tau_j}-\partial_t\varepsilon = R(t)-\partial_t\varepsilon
	\end{align*}
	for all $t\in (t_{j-1},t_j)$.
	Hence, using Duhamel's principle and the continuous semigroup $e^{t\Delta}$, we can express $\rho$ as
	\begin{align*}
		\rho(t) = e^{(t-t_j)\Delta}\rho(t_{j-1})+\int_{t_{j-1}}^te^{(t-s)\Delta}(R(s)-\partial_t\varepsilon(s))\;ds.
	\end{align*}
	Using this expression we can bound
	\begin{align*}
		\norm{\rho(t,\cdot)}_{L^\infty(\Omega)}
		&\leq \norm{\rho(t_{j-1},\cdot)}_{L^\infty(\Omega)}+\int_{t_{j-1}}^t\norm{R(s,\cdot)}_{L^\infty(\Omega)}+\norm{\partial_t\varepsilon(s,\cdot)}_{L^\infty(\Omega)}\;ds\\
		&\leq \norm{\rho(t_{j-1},\cdot)}_{L^\infty(\Omega)}+\eta_T^j+\tau_j\norm{\partial_t\varepsilon}_{L^\infty([t_{j-1},t_j]\times\Omega)}.
	\end{align*}
	It only remains to show $\norm{\partial_t\varepsilon}_{L^\infty([t_{j-1},t_j]\times\Omega)}\leq\dot{\eta}_{S,q}^j$.
	To do this we write
	\begin{align*}
		\partial_t\varepsilon = \tau_j^{-1}(w^{j-1}-w^j-\widehat{u_0}^{j-1}+\widehat{u_0}^j)
	\end{align*}
	and use the same arguments as in Lemma \ref{lemma:EllipticBound} to arrive at
	\begin{multline*}
		\langle\nabla\partial_t\varepsilon,\nabla G \rangle
		\leq \tau_j^{-1}2^{\frac1{p^*}}\left(\sum_{K\in\mathcal{T}}\left(C_{P,p^*}^qh_K^q\norm{A^j-A^{j-1}+\Delta (\widehat{u_0}^j-\widehat{u_0}^{j-1})}_{L^q(K)}^q\right.\right.\\
		\left.\left.+\frac{C_{\text{app},p}^qN_\partial^{q-1}}{2^q}h_K\sum_{\sigma\in\mathcal{E}_K\setminus\mathcal{E}^D}\norm{\left\llbracket\nabla (\widehat{u_0}^j-\widehat{u_0}^{j-1})\right\rrbracket}_{L^q(\sigma)}^q\right)\right)^{\frac1q}\norm{\nabla G}_{L^{p^*}(\Omega)}
	\end{multline*}
	Combining the two estimates finishes the proof.
\end{proof}
\begin{thm}\label{thm:max_heat_bound}
	If Condition \ref{conditions} (C1)-(C3),(C6) and (C7) are satisfied for some exponent $1\leq p^*<\frac{d}{d-1}$, then the $L^\infty(0,T;L^\infty(\Omega))$ norm of $e:=\widehat{u_0}-v_0$ satisfies
	\begin{align*}
		\max_{t\in[0,T]}\norm{e(t,\cdot)}_{L^\infty(\Omega)} \leq \norm{e(0,\cdot)}_{L^\infty(\Omega)}+\sum_{j=1}^J\eta_{T,\infty}^j+\sum_{j=1}^J\tau_j\dot{\eta}_{S,q}^jC_{\text{Green},p^*}+\max_{0\leq j\leq J} \eta_{S,q}^jC_{\text{Green},p^*}.
	\end{align*}
	With $\frac1{p^*}+\frac1q=1$ and the estimators $\eta_{T,\infty}^j,\dot{\eta}_{S,q}^j$ defined in Lemma \ref{lemma:ParabolicBound} and $\eta_{S,q}^j$ defined in Lemma \ref{lemma:EllipticBound}.
\end{thm}
\begin{proof}
	With Lemma \ref{lemma:ParabolicBound} and Lemma \ref{lemma:EllipticBound} it follows by induction over the time steps $j=0,\dots,J$ that
	\begin{align*}
		\max_{t\in[0,T]}\norm{e(t,\cdot)}_{L^\infty(\Omega)} 
		&\leq \max_{t\in[0,T]}\left(\norm{\rho(t,\cdot)}_{L^\infty(\Omega)}+\norm{\varepsilon(t,\cdot)}_{L^\infty(\Omega)}\right)\\
		&\leq \max_{t\in(t_{J-1},t_J]}\norm{\rho(t,\cdot)}_{L^\infty(\Omega)}+\max_{t\in[0,t_{J-1}]}\norm{\rho(t,\cdot)}_{L^\infty(\Omega)}+\max_{0\leq j\leq J}\norm{\varepsilon(t_j,\cdot)}_{L^\infty(\Omega)}\\
		&\leq \norm{\rho(t_{J-1},\cdot)}_{L^\infty(\Omega)}+\eta_{T,\infty}^J+\tau_J\dot{\eta}_S^JC_{\text{Green},p^*}+\max_{t\in[0,t_{j-1}]}\norm{\rho(t,\cdot)}_{L^\infty(\Omega)}+\max_{0\leq j\leq J} \norm{\varepsilon_j}_{L^\infty(\Omega)}\\
		&\leq \norm{e(0,\cdot)}_{L^\infty(\Omega)}+\sum_{j=1}^J\eta_{T,\infty}^j+\sum_{j=1}^J\tau_j\dot{\eta}_S^jC_{\text{Green},p^*}+\max_{0\leq j\leq J} \eta_{S,q}^jC_{\text{Green},p^*}
	\end{align*}
	where we also used
	\begin{align*}
		\norm{\rho(0,\cdot)}_{L^\infty(\Omega)}
		\leq \norm{e(0,\cdot)}_{L^\infty(\Omega)}+\norm{\varepsilon(0,\cdot)}_{L^\infty(\Omega)}
		\leq \norm{e(0,\cdot)}_{L^\infty(\Omega)}+\eta_{S,q}^0C_{\text{Green},p^*}.
	\end{align*}
\end{proof}
\subsection{Gradient error bounds}
We now derive an upper bound for $\norm{\nabla (\widehat{u_0}-v_0)}_{L^2([0,T]\times\Omega)}$.
For this we use a classical stability framework for the diffusion equation.
The proof is similar to that in \cite[Proposition 4.5]{bartels2016}.
The bound on $e$ then only relies on the $L^2([0,T];H^1_D(\Omega)')$ norm of the residual $R_0:=-\sum_{i=1}^nR_i$.
\begin{thm}\label{thm:H1_heat_bound}
	The $L^2$-norm in space and time of $\nabla e=\nabla(\widehat{u_0}-v_0)$ is bounded as follows
	\begin{multline*}
		\sup_{t\in[0,T]}\norm{e(t,\cdot)}_\Omega^2+\int_0^T \norm{\nabla e(t,\cdot)}_{L^2(\Omega)}^2 dt \\
		\leq 2\norm{\widehat{u_0}(0,\cdot)-v_0(0,\cdot)}_{L^2(\Omega)}^2+\sum_{j=0}^J\tau_j\left(\eta_{S,2}^j+C_{F,2,\Gamma_D}\norm{\partial_t\widehat{u_0}^j-\frac{u_{0,h}^j-u_{0,h}^{j-1}}{\tau_j}}_\Omega+\norm{\nabla(\widehat{u_0}^j-\widehat{u_0}^{j-1})}_{\Omega}\right)^2
	\end{multline*}
	with $\eta_{S,2}^j$ defined in Lemma \ref{lemma:EllipticBound}.
\end{thm}
\begin{proof}
	We use the stability estimate from \cite[Proposition 4.5]{bartels2016}.
	Namely, the error $e:= v_0-\widehat{u_0}$ satisfies
	\begin{align*}
		\sup_{t\in[0,T]}\norm{e(t)}_{\Omega}^2+\norm{\nabla e}_{[0,T]\times\Omega}\leq2\norm{e(0)}_{\Omega}^2+\norm{R_0}_*^2.
	\end{align*}
	We now only need to bound $\norm{R_0}_*^2$.
	Let $f\in H^1_D(\Omega)$. For every $t\in(t_{j-1},t_j)$ we have
	\begin{align*}
		\langle R_0(t), f\rangle
		&= \left( \partial_t \widehat{u_0}(t),f\right) +\left( \nabla \widehat{u_0}(t),\nabla f\right)\\
		&= \left( \partial_t \widehat{u_0}-\frac{u_{0,h}^j-u_{0,h}^{j-1}}{\tau_j},f\right)+\frac1{\tau_j}\left( u_{0,h}^j-u_{0,h}^{j-1},f\right) +\left( \ell_j\nabla \widehat{u_0}^j+\ell_{j-1}\widehat{u_0}^{j-1},\nabla f\right)\\
		&=\left( \nabla \left(\widehat{u_0}^j-w^j\right),\nabla f\right)+ \left( \partial_t \widehat{u_0}-\frac{u_{0,h}^j-u_{0,h}^{j-1}}{\tau_j},f\right)+\left( (\ell_j-1)\nabla \widehat{u_0}^j+\ell_{j-1}\nabla \widehat{u_0}^{j-1},\nabla f\right)\\
		&=: I+II+III.
	\end{align*}
	We bound $I$ using Proposition \ref{prop:MorleyOrth}.
	We perform the same steps as in the proof of Lemma \ref{lemma:EllipticBound} to arrive at
	\begin{equation}
		\begin{aligned}\label{eqn:R}
			\left( \nabla \left(\widehat{u_0}^j-w^j\right),\nabla f\right)
			&= \sum_{K\in \mathcal{T}}\left(\int_K (-A^j-\Delta \widehat{u_0}^j)(f-\M_K f) dx\right.
			\left.-\frac12\sum_{\sigma\in\mathcal{E}_K\setminus\mathcal{E}^D}\int_{\sigma} \jump{\nabla \widehat{u_0}^j}(f-\M_K f) dS\right)\\
			&\leq\eta_{S,2}^j\norm{\nabla f}_{\Omega}.
		\end{aligned}
	\end{equation}
	For $II$ we use Hölder's inequality and the Poincaré-Friedrichs inequality (Theorem \ref{thm:friedirchsineq}) to derive
	\begin{align}\label{eqn:R_e}
		\left( \partial_t \widehat{u_0}-\frac{u_{0,h}^j-u_{0,h}^{j-1}}{\tau_j},f\right)
		\leq \norm{\partial_t \widehat{u_0}-\frac{u_{0,h}^j-u_{0,h}^{j-1}}{\tau_j}}_{L^2(\Omega)}\norm{f}_{\Omega}
		\leq C_{F,2,\Gamma_D}\norm{\partial_t \widehat{u_0}-\frac{u_{0,h}^j-u_{0,h}^{j-1}}{\tau_j}}_{\Omega}\norm{\nabla f}_{\Omega}.
	\end{align}
	Notice, that \eqref{eqn:R_e} is not optimal in the sense that we estimate $\norm{\partial_t (\widehat{u_0}-u_{0,h})}_{H^{1}_D(\Omega)'}\lesssim\norm{\partial_t \widehat{u_0}-\frac{u_{0,h}^j-u_{0,h}^{j-1}}{\tau_j}}_\Omega$.
	Since the $L^2$ norm of $\partial_t \widehat{u_0}-\frac{u_{0,h}^j-u_{0,h}^{j-1}}{\tau_j}$ converges linearly as does the rest of the estimator we may use the $L^2$ norm, which is easier to compute than the $H^1_D(\Omega)'$ norm, without changing the convergence behavior of the overall estimator.\\
	For $III$ we use Hölder's inequality
	\begin{equation}
		\begin{aligned}\label{eqn:R_T}
			\left( (\ell_j-1)\nabla \widehat{u_0}^j+\ell_{j-1}\nabla \widehat{u_0}^{j-1},\nabla f\right)
			&\leq \norm{(\ell_j-1)\nabla \widehat{u_0}^j+\ell_{j-1}\nabla \widehat{u_0}^{j-1}}_{\Omega}\norm{\nabla f}_{\Omega}\\
			&\leq \norm{\nabla \widehat{u_0}^j-\nabla \widehat{u_0}^{j-1}}_{\Omega}\norm{\nabla f}_{\Omega}.
		\end{aligned}
	\end{equation}
	Combining estimates \eqref{eqn:R}-\eqref{eqn:R_T} yields the desired result.
\end{proof}
\section{Stability frameworks}\label{section:abstract_error}
We first prove a stability framework for the reduced model \eqref{eqn:reduced} with initial conditions \eqref{eqn:initial} and boundary conditions \eqref{boundaryCond}.
This means we bound the difference $u_i-v_i$ ($i=1,\dots,n$), with $(v_i)_{i=0}^n$ a weak solution in the sense of Theorem \ref{thm:existence} and $(u_i)_{i=0}^n$ a weak solution of the perturbed system \eqref{eqn:fullperturbed}.
This upper bound relies on bounds for the $L^\infty([0,T]\times\Omega)$ norm of $u_0-v_0$ and the $L^2([0,T]\times\Omega)$ norm of $\nabla(u_0-v_0)$, that were derived in Section \ref{section:heat}.\\
After that we present a stability framework for the general model \eqref{eqn:general} with the same boundary conditions \eqref{boundaryCond}.
The upper bound for $u_i-v_i$ ($i=1,\dots,n$) follows in the same way as for the reduced model.
We also need upper bounds for $u_0-v_0$ and $\Phi-\Psi$, since we cannot use the bound from Section \ref{section:heat}.\\
For the stability frameworks, we need the Assumption \ref{ass:perturbed} on the weak solution $(u_0,\dots,u_n,\Phi)$ of the perturbed system \eqref{eqn:fullperturbed}.
Recall $X(q) := L^{\frac{2q}{q-d}}(0,T;L^q(\Omega))$.
In Section \ref{section:residual}, when we derive the a posteriori error estimate the solution to the pertubed system is the reconstruction of the numerical solution.
Hence, we do not need Assumption \ref{ass:perturbed} for the a posteriori error estimate.
\begin{ass}\label{ass:perturbed}
	\begin{enumerate}[(i)]
		\item Assume that the Dirichlet boundary conditions are exactly satisfied by the solution of the perturbed system, i.e.
			\begin{align*}
				u_i = v_i^D\quad \text{on }\Gamma_D.
			\end{align*}
		\item We assume that the solvent concentration fulfills $\nabla u_0\in X(q)$ for some $q>d$.
			If $z\neq 0$, we assume additionally that $F:=\nabla u_0-u_0z\beta\nabla\Phi\in X(q)$ and $\nabla u_i\in L^\infty(0,T;L^{\tilde{q}}(\Omega))$ for some $q,\tilde{q}>2$ for $d=2$ and $q,\tilde{q}\geq 3$ for $d=3$.
	\end{enumerate}
\end{ass}
\begin{remark}
	\begin{enumerate}[(a)]
		\item Assumption \ref{ass:perturbed} (i) states, that the perturbed system satisfies the same Dirichlet boundary conditions as the original system.
			If we did not impose Assumption (i), we would have to bound
			\begin{align*}
				\langle u_0\partial_n u_i-u_i\partial_n u_0-v_0\partial_n v_i+v_i\partial_n v_0,u_i-v_i\rangle_{H^{-\frac12}(\Gamma_D),H^{\frac12}(\Gamma_D)}
			\end{align*}
			in the proof of Theorem \ref{thm:abstract_error} and Theorem \ref{thm:abstract_general}.
			It is unclear how this can be done.
			If the boundary conditions satisfy (C3), then the reconstruction of the numerical solution from Section \ref{section:scheme} satisfies Assumption \ref{ass:perturbed} (i) automatically.
		\item We do not assume that the Neumann boundary condition is exactly satisfied.
			The perturbed Neumann boundary conditions are incorporated in the right hand side $R_i\in H^1_D(\Omega)$. 
			Hence, we do not see the boundary error explicitly in this setting.
		\item We need the additional regularity assumptions in (ii) for the full model, since we cannot utilize the classical maximal regularity results for the diffusion equation.
			Since the reconstruction is composed of continuous piecewise polynomial function, Assumption \ref{ass:perturbed} (ii) is automatically satisfied.
	\end{enumerate}
\end{remark}
\subsection{Stability framework for the reduced system}\label{subsec:stability_reduced}
We now establish a stability framework for the simplified ion transport model.
Let, in this subsection, $(v_i)_{i=0,\dots,n}$ be a weak solution of \eqref{eqn:reduced} and $(u_i)_{i=0,\dots,n}$ be a weak solution of the perturbed system
\begin{align}\label{eqn:simpert}
	&\begin{aligned}
	\partial_t u_i - \operatorname{div}(u_0\nabla u_i-u_0\nabla u_i) &= R_i\quad i=1,\dots,n\\
	\sum_{i=0}^n u_i &= 1
	\end{aligned}\quad \text{in }\Omega\times(0,T)
\end{align}
with $R_i\in L^2(0,T;H^1_D(\Omega)')$ and $(u_0,\dots,u_n)$ satisfying Assumption \ref{ass:perturbed}.
For this stability framework, it does not matter how the solution to the perturbed system is obtained, as long as it has the required regularity and satisfies Assumption \ref{ass:perturbed}.
This stability framework can be used with other numerical methods or different reconstructions, since we only impose minimal conditions on the residuals $R_i$ and the regularity conditions on $u_0,\dots,u_n$.\\
To obtain the bound, we subtract the weak formulations for \eqref{eqn:reduced} and \eqref{eqn:simpert} and test with $u_i-v_i$.
In the end, we use the Gronwall lemma.
Note, that according to \cite[Section 5.9, Theorem 3.9]{evans2010} $v_i\in C([0,T],L^2(\Omega))$ for all $i=0,\dots,n$ and $T>0$.
The positivity of the heat semigroup $(e^{t\Delta})_{t\geq0}$ implies that $\gamma<v_0(t,x)$ for all $(t,x)\in[0,T]\times\Omega$ (see \cite[Section 4.1]{ouhabaz2009}).
\begin{thm}\label{thm:abstract_error}
	Let the data satisfy Condition \ref{conditions} (C1) and (C2).
	Let $(v_0,\dots,v_n)$ be a weak solution of \eqref{equations}-\eqref{eqn:initial} and $(u_0,\dots,u_n)$ be a weak solution of \eqref{eqn:simpert} with the boundary and initial conditions \eqref{boundaryCond}-\eqref{eqn:initial}.
	Under Assumption \ref{ass:perturbed}, the difference $u_i-v_i$ for $i=1,\dots,n$ is bounded as follows
	\begin{multline*}
		\max_{t\in[0,T]} \norm{u_i-v_i}_{\Omega}^2+\int_0^{T}\norm{\sqrt{v_0}\nabla(u_i-v_i)}_{\Omega}^2\,dt
		\leq\\
		\left(2\norm{u_i^0-v_i^0}^2_{\Omega}
		+\frac{12}\gamma\norm{v_0-u_0}_{L^\infty([0,T]\times\Omega)}^2\norm{\nabla u_i}_{[0,T]\times\Omega}^2
		+\frac{12}\gamma\norm{\nabla(u_0-v_0)}_{[0,T]\times\Omega}^2
		+\frac{12}\gamma \norm{R_i}_{*}^2 \right)\\
		\exp\left(2C_G^{\frac{2}{1-\theta}}(1+C_{F,2,\Gamma_D})^{\frac{\theta}{1-\theta}}\frac{\mu}{\gamma^{\frac{1+\theta}{1-\theta}}}\norm{\nabla u_0}_{X(q)}^{\frac{2}{1-\theta}}\right).
	\end{multline*}
	with $\theta = \frac d2-\frac dp$ and $\mu= \frac{1-\theta}{2}\left(\frac{1}{2(1+\theta)}\right)^\frac{\theta+1}{\theta-1}$.
\end{thm}
\begin{proof}
	Subtracting the weak formulations for $u_i$ and $v_i$ and testing with $u_i-v_i$ yields
	\begin{align*}
		\frac12\frac{d}{dt}\int_\Omega(u_i-v_i)^2\,dx
		= &-\int_\Omega \left(u_0\nabla u_i-u_i\nabla u_0-v_0\nabla v_i+v_i\nabla v_0\right)\cdot\left(\nabla (u_i-v_i)\right)\,dx\\
		&+\langle R_i,u_i-v_i\rangle.
	\end{align*}
	Rearranging terms and using Hölder's inequality yields
	\begin{equation}\label{eqn:abstract_err_gron}
		\begin{aligned}
		\frac12\frac{d}{dt}\norm{u_i-v_i}_{\Omega}^2
		\leq &-\norm{\sqrt{v_0}\nabla(u_i-v_i)}_{\Omega}^2+\norm{v_0-u_0}_{L^\infty(\Omega)}\norm{\nabla u_i}_{\Omega}\norm{\nabla(u_i-v_i)}_{\Omega}\\
		&+\norm{v_i}_{L^\infty(\Omega)}\norm{\nabla(u_0-v_0)}_{\Omega}\norm{\nabla(u_i-v_i)}_{\Omega}\\
		&+\norm{u_i-v_i}_{L^p(\Omega)}\norm{\nabla u_0}_{L^q(\Omega)}\norm{\nabla(u_i-v_i)}_{\Omega}\\
		&+\norm{R_i}_{H^1_D(\Omega)'}\norm{\nabla (u_i-v_i)}_{\Omega},
		\end{aligned}
	\end{equation}
	with $\frac1p+\frac1q=\frac12$ and $q>d$.
	To bound the $L^p$ norm of $u_i-v_i$ we use the Gagliardo-Nirenberg inequality (Theorem \ref{thm:Nirenberg}) with $\theta=\frac d2-\frac dp$
	\begin{equation}
	\begin{aligned}
		\norm{u_i-v_i}_{L^p(\Omega)}\label{eq:Nirenberg}
		&\leq C_G\norm{u_i-v_i}_{H^1(\Omega)}^\theta\norm{u_i-v_i}_{\Omega}^{1-\theta}\\
		&\leq C_G\left(1+C_{F,2,\Gamma_D}^2\right)^{\frac{\theta}{2}}\norm{u_i-v_i}_{\Omega}^{1-\theta}\norm{\nabla(u_i-v_i)}_{\Omega}^{\theta}.
	\end{aligned}
	\end{equation}
	Using \eqref{eqn:abstract_err_gron},\eqref{eq:Nirenberg} and Young's inequality we arrive at
	\begin{align*}
		\frac12\frac{d}{dt}\norm{u_i-v_i}_{\Omega}^2
		\leq &-\frac12\norm{\sqrt{v_0}\nabla(u_i-v_i)}_{\Omega}^2
		+\frac6{\gamma}\norm{v_0-u_0}_{L^\infty(\Omega)}^2\norm{\nabla u_i}_{\Omega}^2
		+\frac6{\gamma}\norm{v_i}_{L^\infty(\Omega)}^2\norm{\nabla(u_0-v_0)}_{\Omega}^2\\
		&+C_G^\frac2{1-\theta}(1+C_{F,2,\Gamma_D}^2)^{\frac{\theta}{1-\theta}}\frac{\mu}{\gamma^{\frac{1+\theta}{1-\theta}}}\norm{\nabla u_0}_{L^q(\Omega)}^{\frac{2}{1-\theta}}\norm{u_i-v_i}_{\Omega}^2
		+\frac6{\gamma}\norm{R_i}_{H^1_D(\Omega)'}^2,
	\end{align*}
	with $\mu := \frac{1-\theta}{2}\left(\frac{1}{2(1+\theta)}\right)^\frac{\theta+1}{\theta-1}$.
	The claim then follows with the classical Gronwall inequality.
	\begin{multline*}
		\max_{t\in[0,T]} \norm{u_i-v_i}_{\Omega}^2+\int_0^{T}\norm{\sqrt{v_0}\nabla(u_i-v_i)}_{\Omega}^2\,dt
		\leq\\
		\left(2\norm{u_i^0-v_i^0}^2_{\Omega}
		+\frac{12}\gamma\norm{v_0-u_0}_{L^\infty([0,T]\times\Omega)}^2\norm{\nabla u_i}_{[0,T]\times\Omega}^2
		+\frac{12}\gamma\norm{\nabla(u_0-v_0)}_{[0,T]\times\Omega}^2
		+\frac{12}\gamma \norm{R_i}_{*}^2 \right)\\
		\exp\left(2C_G^{\frac{2}{1-\theta}}(1+C_{F,2,\Gamma_D}^2)^{\frac{\theta}{1-\theta}}\frac{\mu}{\gamma^{\frac{1+\theta}{1-\theta}}}\norm{\nabla u_0}_{X(q)}^{\frac{2}{1-\theta}}\right).
	\end{multline*}
\end{proof}
\subsection{A stability framework for the general model}\label{abserror:general}
We now derive a stability framework for the general model.
Let $(u_0,\dots,u_n,\Phi)$ be a weak solution of the perturbed system \eqref{eqn:fullperturbed} satisfying Assumption \ref{assumptions}.
Let $(v_0,\dots,v_n,\Psi)$ be a weak solution for the general model \eqref{eqn:general} in the sense of Theorem \ref{thm:existence}.
The main differences, compared to the previous section, are that we need to control the difference of the electric potential and that the solvent no longer solves the diffusion equation and we do not have a good bound for the difference $v_0-u_0$ in $L^\infty([0,T]\times\Omega)$.\\
We first derive a stability framework for the electric potential.
To do so, we note that the electric potential solves the Poisson equation.
Therefore, we can use the classical stability estimate for the Poisson equation obtained by the Poincar\'e inequality
\begin{align}\label{eqn:electricpot}
	\norm{\nabla (\Phi-\Psi)}_{\Omega} \leq C_{F,2,\Gamma_D}\frac{\abs{z}}{\lambda^2}\norm{u_0-v_0}_{\Omega}+\frac1{\lambda^2}\norm{R_\Phi}_{H^1_D(\Omega)'}\quad\forall t\in[0,T].
\end{align}
For the ion concentrations we can perform similar steps as in Theorem \ref{thm:abstract_error} with $F:=\nabla u_0-z\beta u_0\nabla \Phi$ and $G:=\nabla v_0-z\beta v_0\nabla \Psi$ instead of $\nabla u_0$ and $\nabla v_0$ respectively.
Furthermore, to estimate $\norm{(u_0-v_0)\nabla u_i}_{[0,T]\times\Omega}$ without using $\norm{u_0-v_0}_{L^\infty([0,T]\times\Omega)}$ we utilize the Sobolev inequality
\begin{align*}
	\int_0^T \int_\Omega (u_0-v_0)\nabla u_i\cdot\nabla (u_i-v_i)\,dx\,dt
	&\leq \int_0^T\norm{u_0-v_0}_{L^p(\Omega)}\norm{\nabla u_i}_{L^{\tilde{q}}(\Omega)}\norm{\nabla(u_i-v_i)}_{\Omega}\,dt\\
	&\leq C_S\norm{u_0-v_0}_{L^2(0,T;H^1(\Omega))}\norm{\nabla u_i}_{L^\infty(0,T;L^{\tilde{q}}(\Omega))}\norm{\nabla (u_i-v_i)}_{L^2(0,T;L^2(\Omega))},
\end{align*}
with $\frac1p+\frac1{\tilde{q}}=\frac12$ and $\tilde{q}>2$ for $d=2$ and $\tilde{q}\geq 3$ for $d=3$.
Therefore, here we need Assumption \ref{ass:perturbed} (ii), i.e. $\nabla u_i\in L^\infty(0,T;L^{\tilde{q}}(\Omega))$ instead of $L^2(0,T;L^2(\Omega))$.
This gives us the estimate
\begin{multline*}
	\max_{t\in[0,T]} \norm{u_i-v_i}_{\Omega}^2+\int_0^{T}\norm{\sqrt{v_0}\nabla(u_i-v_i)}_{\Omega}^2\,dt
	\leq\\
	\left(2\norm{u_i^0-v_i^0}^2_{\Omega}
	+\frac{12}\gamma\norm{\nabla (v_0-u_0)}_{[0,T]\times\Omega}^2\norm{\nabla u_i}_{L^\infty(0,T;L^q(\Omega))}^2
	+\frac{12}\gamma\norm{F-G}_{[0,T]\times\Omega}^2
	+\frac{12}\gamma \norm{R_i}_{*}^2 \right)\\
	\exp\left(2C_G^{\frac{2}{1-\theta}}(1+C_{F,2,\Gamma_D}^2)^{\frac{\theta}{1-\theta}}\frac{\mu}{\gamma^{\frac{1+\theta}{1-\theta}}}\norm{F}_{X(q)}^{\frac{2}{1-\theta}}\right)
\end{multline*}
with $\theta$ and $\mu$ from Theorem \ref{thm:abstract_error}.
Furthermore, we can bound the difference $F-G$ by the use of Sobolev inequality, Poincar\'e-Friedrichs inequality and equation \eqref{eqn:electricpot}, Theorem \ref{thm:friedirchsineq}
\begin{align*}
	\norm{F-G}_{[0,T]\times\Omega}^2
	&\leq\int_0^T\left(\norm{\nabla (u_0-v_0)}_{\Omega}+\abs{z\beta}\norm{(u_0-v_0)\nabla\Phi}_{\Omega}+\abs{z\beta}\norm{v_0\nabla(\Phi-\Psi)}_{\Omega}\right)^2\,dt\\
	&\leq\int_0^T\left(\norm{\nabla (u_0-v_0)}_{\Omega}+\abs{z\beta}\norm{u_0-v_0}_{L^p(\Omega)}\norm{\nabla\Phi}_{L^{\tilde{q}}(\Omega)}+\abs{z\beta}\norm{\nabla(\Phi-\Psi)}_{\Omega}\right)^2\,dt\\
	&\leq2\norm{\nabla (u_0-v_0)}_{[0,T]\times\Omega}^2\left(1+C_S\sqrt{1+C_{F,2,\Gamma_D}^2}\abs{z\beta}\norm{\nabla\Phi}_{L^\infty(0,T;L^{\tilde{q}}(\Omega))} +C_{F,2,\Gamma_{D}}\frac{\abs{z}^2\beta}{\lambda^2}\right)^2\\
	&\;+2\frac{\abs{z\beta}^2}{\lambda^2}\norm{R_\Phi}_{*}^2,
\end{align*}
where we used $0\leq v_0\leq 1$ in the second estimate.
We now turn our attention to the solvent concentration.
Since the solvent concentration does not solve the diffusion equation in the general model, we cannot use the classical stability framework here.
To control the difference $u_0-v_0$, we test the weak formulation for the solvent concentration equation \eqref{eqn:solvent} with $u_0-v_0$ and use Hölder's inequality and $0\leq v_0\leq1$ to arrive at
\begin{align*}
	\frac12\frac{d}{dt}\norm{u_0-v_0}_{\Omega}^2
	&\leq -\norm{\nabla (u_0-v_0)}_{\Omega}^2+\abs{z\beta}\norm{u_0-v_0}_{L^p(\Omega)}\norm{\nabla\Phi}_{L^{q}(\Omega)}\norm{\nabla(u_0-v_0)}_{\Omega}\\
	&+\frac{\abs{z\beta}}4\norm{\nabla(\Phi-\Psi)}_{\Omega}\norm{\nabla(u_0-v_0)}_{\Omega}+\norm{R_0}_{H^1_D(\Omega)'}\norm{\nabla(u_0-v_0)}_\Omega.
\end{align*}
The right hand side contains the $L^p$ norm of the difference $u_0-v_0$.
We estimate this term using the Gagliardo-Nirenberg inequality
with $\theta =\frac d2-\frac dp$ similar to equation \eqref{eq:Nirenberg} and use equation \eqref{eqn:electricpot} to arrive at
\begin{align*}
	\frac12\frac{d}{dt}\norm{u_0-v_0}_{\Omega}^2
	&\leq -\frac12\norm{\nabla (u_0-v_0)}_{\Omega}^2
	+C_G^{\frac2{1-\theta}}\mu\norm{u_0-v_0}^2_{\Omega}\norm{\nabla\Phi}^{\frac2{1-\theta}}_{L^q(\Omega)}\\
	&+\frac{\abs{z\beta}^2}{8}\left(C_{F,2,\Gamma_D}^2\frac{\abs{z}^2}{\lambda^4}\norm{u_0-v_0}^2_{\Omega}+\norm{R_\Phi}^2_{H^1_D(\Omega)'}\right)+2\norm{R_0}_{H^1_D(\Omega)'}^2.
\end{align*}
Applying Gronwall's inequality yields
\begin{multline*}
	\max_{t\in[0,T]}\norm{u_0-v_0}_{\Omega}^2+\int_0^T\norm{\nabla (u_0-v_0)}_{\Omega}^2\,dt
	\leq \left(2\norm{u_0^0-v_0^0}_{\Omega}^2
	+\frac{\abs{z\beta}^2}{4}\norm{R_\Phi}_{*}^2+4\norm{R_0}_{*}^2\right)\\
	\exp\left(2C_G^{\frac{2}{1-\theta}}(1+C_{F,2,\Gamma_D}^2)^{\frac{\theta}{\theta-1}}\mu\norm{\nabla\widehat{\Phi}}_{X(q)}^{\frac{2}{1-\theta}}+C_{F,2,\Gamma_D}^2T\frac{\abs{z}^4\beta^2}{8\lambda^4}\right).
\end{multline*}
We now summarize the inequalities from above in one theorem.
\begin{thm}\label{thm:abstract_general}
	Let the data satisfy Condition \ref{conditions} (C1) and (C2).
	Let Assumption \ref{assumptions} and \ref{ass:perturbed} hold.
	Let $(u_0,\dots,u_n, \Phi)$ be a weak solution of \eqref{eqn:fullperturbed} and $(v_0,\dots,v_n,\Psi)$ be a weak solution of \eqref{equations}-\eqref{eqn:initial}.
	The electric potentials $\Phi$ and $\Psi$ satisfy
	\begin{align*}
		\norm{\nabla (\Phi-\Psi)}_{\Omega} \leq C_{F,2,\Gamma_D}\frac{\abs{z}}{\lambda^2}\norm{u_0-v_0}_{\Omega}+\norm{R_\Phi}_{H^1_D(\Omega)'}\quad\forall t\in[0,T].
	\end{align*}
	The difference $u_0-v_0$ satisfies
	\begin{multline*}
		\max_{t\in[0,T]}\norm{u_0-v_0}_{\Omega}^2+\int_0^T\norm{\nabla (u_0-v_0)}_{\Omega}^2\,dt
		\leq \left(2\norm{u_0^0-v_0^0}_{\Omega}^2+\frac{\abs{z\beta}^2}{4}\norm{R_\Phi}_{*}^2+4\norm{R_0}_{*}^2\right)\\
		\exp\left(2C_G^{\frac{2}{1-\theta}}(1+C_{F,2,\Gamma_D}^2)^{\frac{\theta}{1-\theta}}\mu\norm{\nabla\widehat{\Phi}}_{X(q)}^{\frac{2}{1-\theta}}+C_{F,2,\Gamma_D}^2T\frac{\abs{z}^4\beta^2}{8\lambda^4}\right).
	\end{multline*}
	The difference $u_i-v_i$ for $i=1,\dots,n$ satisfies
	\begin{multline*}
		\max_{t\in[0,T]} \norm{u_i-v_i}_{\Omega}^2+\int_0^{T}\norm{\sqrt{v_0}\nabla(u_i-v_i)}_{\Omega}^2\,dt
		\leq\\
		\left(\norm{u_i^0-v_i^0}^2_{\Omega}
		+\frac{5}\gamma\norm{\nabla (v_0-u_0)}_{[0,T]\times\Omega}^2\norm{\nabla u_i}_{L^\infty(0,T;L^{\tilde{q}}(\Omega))}^2
		+\frac{5}\gamma\norm{F-G}_{[0,T]\times\Omega}^2
		+\frac{5}\gamma\norm{R_i}_{*}^2 \right)\\
		\exp\left(2C_G^{\frac{2}{1-\theta}}(1+C_{F,2,\Gamma_D}^2)^{\frac{\theta}{1-\theta}}\frac\mu{\gamma^{\frac{1+\theta}{1-\theta}}}\norm{F}_{X(q)}^{\frac{2}{1-\theta}}\right),
	\end{multline*}
	with $\tilde{q},q>2$ for $d=2$ and $\tilde{q},q\geq 3$ for $d=3$ and $\theta=\frac d2-\frac dp,\mu=\frac{1-\theta}{2}\left(\frac{1}{2(1+\theta)}\right)^\frac{\theta+1}{\theta-1}$.
\end{thm}
Let us comment on our stability theory, Theorem \ref{thm:abstract_general}, and its relation to commonly used methods for stability and uniqueness of cross diffusion systems.
\begin{remark}\label{Whyl2}
	\textcite{gerstenmayer2018} showed a uniqueness result for the general model \eqref{eqn:general} using the Gajewski method.
	They use the entropy functional
	\begin{align*}
		H[v] = \sum_{i=0}^n\int_\Omega v_i(\log(v_i)-1)\,dx
	\end{align*}
	for a solution $v=(v_0,\dots,v_n)$ of \eqref{eqn:general}.
	Following the route of Gajewski (see \cite{gajewski1994}), this leads to the metric
	\begin{align*}
		d(u,v) = H[u]+H[v]-2H\left[\frac{u+v}{2}\right]
	\end{align*}
	for two solutions $v=(v_0,\dots,v_n)$ and $u=(u_0,\dots,u_n)$.
	This metric can be seen as a symmetrization of the often used relative entropy (see \cite[Remark 4]{chen2018}).\\
	This entropy structure is often used to prove weak uniqueness and weak-strong uniqueness results for cross-diffusion systems see also \cite{burger2010,zamponi2015,berendsen2020,jungel2015,hopf2022}.
	Instead, we use the $L^2$ norm to establish a stability framework.
	Using the $L^2$ norm in this context is advantageous.
	The $L^2$ metric is weaker in the sense that
	\begin{align*}
		\norm{u-v}_{L^2(\Omega)}^2\leq \norm{\sqrt{u}-\sqrt{v}}_{L^2(\Omega)}^2\leq d(u,v)
	\end{align*}
	for two solution $v=(v_0,\dots,v_n)$ and $u=(u_0,\dots,u_n)$.
	This allows us to obtain an a posteriori error estimate.
	More precisely, in the Gajewski metric setting the equation for $u_i$ is tested with $\log\left(\frac{2u_i}{u_i+v_i}\right)$.
	It is unclear to us how to bound the residual tested with this function, since $\log\left(\frac{2u_i}{u_i+v_i}\right)\notin H^1(\Omega)$ for $u_i=0$ outside of a set of measure zero.
	Whether this problem can be overcome by the use of another numerical method or by reconstructing in a different way is an interesting question beyond the scope of this paper.\\
	One reason the Gajewski approach is used in the weak-strong uniqueness proof is that \textcite{gerstenmayer2018} aim to avoid the assumption that $\nabla u_0\in L^\infty(0,T;L^q(\Omega))$ for one solution.
	In the setting of a posteriori analysis, the role of $u_0$ is played by a continuous piecewise polynomial function and therefore $\nabla u_0\in L^\infty(0,T;L^q(\Omega))$ is satisfied.
\end{remark}
\section{A posteriori error estimator}\label{section:residual}
We now establish a posteriori error estimates using the stability frameworks from Section \ref{section:abstract_error} and the reconstruction defined in Section \ref{section:scheme}.
The only thing left is to establish a computable upper bound for the residuals $R_i$ in the $L^2(0,T;H^1_D(\Omega)')$-norm.
To achieve this we use the properties of the Morley type reconstruction from Section \ref{subsec:Morley}.
We prove the first key property in Lemma \ref{orth_recon}.
The residual for the reduced model is given by
\begin{align*}
	\langle R_i,f\rangle
	= \int_\Omega \partial_t \widehat{u_i}\,f\,dx +\int_\Omega \left(\widehat{u_0}\nabla \widehat{u_i}-\widehat{u_i}\nabla \widehat{u_0}\right)\cdot\nabla f\,dx
\end{align*}
for all $f\in H^1_D(\Omega)$.
We first present the bound of the residual for the reduced model \eqref{eqn:reduced} in detail.
Together with the stability framework and bounds on the difference $u_0-v_0$ from Section \ref{section:heat}, we can state the a posteriori error estimate for the difference $u_i-v_i$ ($i=1,\dots,n$).\\
The bound for the general model can be deduced in a similar fashion. 
We state the residual bound and a posteriori estimate for the general model in Section \ref{subsec:general_residual}.

\subsection{Bounding the $H^1_D(\Omega)'$ Norm of the residual for the reduced model}\label{Section:Residual_bound}
Lemma \ref{orth_recon} is crucial to estimate the element residual in Theorem \ref{thm:Residual_bound}.
Since we do not reconstruct with respect to convection, we get an error term on the right hand side, that is not present in \cite[Lemma 5.2]{nicaise2006}.
\begin{lemma}\label{orth_recon}
	The reconstruction $\widehat{u_i}$ satisfies
	\begin{align*}
		\int_K \partial_t u_{i,h}^j dx -\int_K \operatorname{div}(\widehat{u_0}^j\nabla \widehat{u_i}^j-\widehat{u_i}^j\nabla \widehat{u_0}^j) dx=-\sum_{\sigma\in\mathcal{E}_K}\left(u_{i,\sigma}^jF_{\sigma}^K(u_0^j)-\int_\sigma \widehat{u_i}^j\nabla \widehat{u_0}^j\cdot n_{K,\sigma}\, dS\right)
	\end{align*}
	for all $K\in\mathcal{T}$, $i=1,\dots,n$ and $j=1,\dots,J$.
\end{lemma}
\begin{proof}
	We first use Gauss' theorem on every element $K\in\mathcal{T}$ to arrive at
	\begin{align*}
		\int_K \partial_t u_{i,h}^j dx -\int_K \operatorname{div}(\widehat{u_0}^j\nabla \widehat{u_i}^j-\widehat{u_i}^j\nabla \widehat{u_0}^j) dx
		&= \int_K \partial_t u_{i,h}^j dx -\sum_{\sigma\in\mathcal{E}_K}\int_\sigma \left(\widehat{u_0}^j\nabla \widehat{u_i}^j-\widehat{u_i}^j\nabla \widehat{u_0}^j\right)\cdot n_{K,\sigma} dS\\
		&= \int_K \partial_t u_{i,h}^j dx
		-\sum_{\sigma\in\mathcal{E}_K}\left(u_{0,\sigma}^jF^K_{\sigma}(u_i^j)-u_{i,\sigma}^jF_{\sigma}^K(u_0^j)\right)\\
		&-\sum_{\sigma\in\mathcal{E}_K}\left(u_{i,\sigma}^jF_{\sigma}^K(u_0^j)-\int_\sigma \widehat{u_i}^j\nabla \widehat{u_0}^j\cdot n_{K,\sigma}\, dS\right).
	\end{align*}
	The claim now follows from the definition of the numerical scheme.
\end{proof}
\begin{remark}
	Notice that we used $\partial_t u_{i,h}^j$ in Lemma \ref{orth_recon} with $u_{i,h}$ the piecewiese constant in space finite volume solution and not the reconstruction $\widehat{u_i}$.
	The reason for this is, that the finite volume solution solves the equation
	\begin{align*}
		m(K)\partial_t u_{i,h}^j = \sum_{\sigma\in\mathcal{E}_K} u_{0,\sigma}^j F_{\sigma}^K(u_i^j)-u_{i,\sigma}^j F_\sigma^K(u_0^j)
	\end{align*}
	for all $i=1,\dots,n$, $j=0,\dots,J$ and $K\in\mathcal{T}$.
	The reconstruction is done in such a way, that we treat $\partial_t u_{i,h}^j$ as a right hand side.
	Therefore, similar to \cite[Lemma 5.2]{nicaise2006} we have to use $\partial_t u_{i,h}^j$ instead of $\partial_t \widehat{u_i}^j$.
\end{remark}
We can now bound the $H^1_D(\Omega)'$ norm of $R_i$.
\begin{thm}\label{thm:Residual_bound}
	The residual $R_i$ is bounded in $H^1_D(\Omega)'$ by
	\begin{align*}
		\norm{R_i(t)}_{H^1_D(\Omega)'} &\leq R_{S,i}^j+R_{T,i}^j+R_{R,i}^j\quad\forall t\in[t_{j-1},t_j]
	\end{align*}
	where the spacial bound $R_{S,i}^j$, temporal bound $R_{T,i}^j$ and reconstruction bound $R_{R,i}^j$ are given by
	\begin{align*}
		R_{S,i}^j &= \sqrt{2}\left(\sum_{K\in\mathcal{T}}\left(h_K^2C_{P,2}^2\norm{\partial_t u_{i,h}^j-\operatorname{div}\left(\widehat{u_0}^j\nabla \widehat{u_i}^j-\widehat{u_i}^j\nabla \widehat{u_0}^j\right)}_K^2\right.\right.\\
		&\left.\left.+\frac{N_\partial}{4} C_{\text{app},2}^2h_K\sum_{\sigma\in\mathcal{E}_K\setminus\mathcal{E}^D}\norm{\llbracket \widehat{u_0}^j\nabla \widehat{u_i}^j-\widehat{u_i}^j\nabla \widehat{u_0}^j\rrbracket}_\sigma^2\right)\right)^\frac12\\
		R_{T,i}^j &= \norm{\widehat{u_0}^j\nabla \widehat{u_i}^j-\widehat{u_i}^j\nabla \widehat{u_0}^j-\widehat{u_0}^{j-1}\nabla \widehat{u_i}^{j-1}+\widehat{u_i}^{j-1}\nabla \widehat{u_0}^{j-1}}_\Omega\\
		R_{R,i}^j &=C_{F,2,\Gamma_D}\norm{\partial_t (\widehat{u_i}^j-u_{i,h}^j)}_{\Omega} +\frac{C_{\text{app},2}N_\partial^\frac12}{2}\left(\sum_{K\in\mathcal{E}}\sum_{\sigma\in\mathcal{E}_K\setminus\mathcal{E}^N}h_\sigma\norm{\widehat{u_i}^j\dc{\nabla \widehat{u_0}^j}\cdot n_{K,\sigma}-u_{i,\sigma}\frac{F_{\sigma}^K(u_0^j)}{\abs{\sigma}}}_{\sigma}^2\right)^\frac12.
	\end{align*}
\end{thm}
\begin{proof}
	Let $f\in H^1_D(\Omega)$.
	We split the residual into a temporal part $II$ and the rest $I$
	\begin{multline*}
		\langle R[u]_i,f\rangle
		= \underbrace{\langle \partial_t \widehat{u_i},f\rangle +\langle \widehat{u_0}^j\nabla \widehat{u_i}^j-\widehat{u_i}^j\nabla \widehat{u_0}^j,\nabla f\rangle}_{=:I}\\
		+\underbrace{\langle \widehat{u_0}\nabla \widehat{u_i}-\widehat{u_i}\nabla \widehat{u_0}-\widehat{u_0}^j\nabla \widehat{u_i}^j+\widehat{u_i}^j\nabla \widehat{u_0}^j,\nabla f\rangle}_{=:II}\\
	\end{multline*}
	for $t\in(t_{j-1},t_{j})$.
	For the first part $I$, we use Gauss' theorem on every triangle $K\in\mathcal{T}$ and Lemma \ref{orth_recon} to arrive at
	\begin{align*}
		I &=\int_\Omega \partial_t (\widehat{u_i}-u_{i,h})f\,dx+\sum_{K\in\mathcal{T}} \int_K (\partial_t u_{i,h} - \operatorname{div}(\widehat{u_0}^j\nabla \widehat{u_i}^j-\widehat{u_i}^j\nabla \widehat{u_0}^j))f\, dx\\
		&+\sum_{\sigma\in\mathcal{E}_K\setminus\mathcal{E}^D}\int_\sigma (\widehat{u_0}\nabla \widehat{u_i} - \widehat{u_i}\nabla \widehat{u_0})\cdot n_{K,\sigma} f\, dS\\
		&=\underbrace{\int_\Omega \partial_t (\widehat{u_i}-u_{i,h})f\,dx}_{=:I_1}+\underbrace{\sum_{K\in\mathcal{T}} \int_K (\partial_t u_{i,h} - \operatorname{div}(\widehat{u_0}^j\nabla \widehat{u_i}^j-\widehat{u_i}^j\nabla \widehat{u_0}^j))(f-\M_Kf) dx}_{=:I_2}\\
		&\underbrace{-\sum_{K\in\mathcal{T}}\sum_{\sigma\in\mathcal{E}_K}\left(u_{i,\sigma}^jF_{\sigma}^K(u_0^j)-\int_\sigma \widehat{u_i}^j\nabla \widehat{u_0}^j\cdot n_{K,\sigma}\,dS\right)\M_Kf +\sum_{\sigma\in\mathcal{E}_K\setminus\mathcal{E}^D}\int_\sigma\left(\widehat{u_0}^j\nabla \widehat{u_i}^j -\widehat{u_i}^j\nabla \widehat{u_0}^j\right)\cdot n_{K,\sigma} f\, dS}_{=:I_3} 
	\end{align*}
	For $I_1$, we use Hölder's inequality and Poincar\'e-Friedrichs inequality (Theorem \ref{thm:friedirchsineq}) to arrive at
	\begin{align}\label{eqn:bound_I_1}
		I_1\leq C_{F,2,\Gamma_D}\norm{\partial_t(\widehat{u_i}-u_{i,h})}_\Omega\norm{\nabla f}_\Omega.
	\end{align}
	Notice, that similar to the proof of Theorem \ref{thm:H1_heat_bound} we estimated the $H^1_D(\Omega)'$-norm of $\partial_t(\widehat{u_i}-u_{i,h})$ with the $L^2$-norm.
	Similarly, the $L^2$-norm converges linearly in $h$ and therefore we use the $L^2$-norm rather than the $H^1_D(\Omega)'$-norm.\\
	For $I_2$ we also use Hölder's inequality and Poincar\'e inequality on every element
	\begin{align}\label{eqn:bound_I_2}
		I_2 
		\leq C_{\text{P,2}}\sum_{K\in\mathcal{T}}h_K\norm{\partial_t u_{i,h} - \operatorname{div}\left(\widehat{u_0}^j\nabla \widehat{u_i}^j-\widehat{u_i}^j\nabla \widehat{u_0}^j\right)}_K\norm{\nabla f}_K.
	\end{align}
	For the last part $I_3$, we use the properties of the reconstruction to derive that
	\begin{align}\label{eqn:jump_i}
		\sum_{K\in\mathcal{T}}\sum_{\sigma\in\mathcal{E}_K\setminus\mathcal{E}^D}\int_\sigma\llbracket \widehat{u_0}^j\nabla \widehat{u_i}^j\rrbracket f\,dS
		=\sum_{K\in\mathcal{T}}\sum_{\sigma\in\mathcal{E}_K\setminus\mathcal{E}^D}\int_\sigma\llbracket \widehat{u_0}^j\nabla \widehat{u_i}^j\rrbracket (f-\M_K f)\,dS
	\end{align}
	and
	\begin{align}\label{eqn:jump_av_equation}
		\int_\sigma \widehat{u_i}^j\nabla\widehat{u_0}^j\cdot n_{K,\sigma}(f-\M_Kf)\,dS
		= \int_\sigma \left(\frac12\jump{\widehat{u_i}^j\nabla\widehat{u_0}^j}+\widehat{u_i}^j\dc{\nabla \widehat{u_0}^j}\right)(f-\M_Kf)\,dS.
	\end{align}
	Furthermore, with the conservation of mass and $f\in H^1_D(\Omega)$ follows
	\begin{align}\label{eqn:conservation_residual}
		\sum_{K\in\mathcal{T}}\sum_{\sigma\in\mathcal{E}_K\setminus\mathcal{E}^N} \int_\sigma u_{i,\sigma}^j\frac{F_\sigma^K(u_0^j)}{\abs{\sigma}}f\,dS = 0
	\end{align}
	Using \eqref{eqn:jump_av_equation} and \eqref{eqn:jump_i} for the second equality and \eqref{eqn:conservation_residual} in the last equality yields
	\begin{align*}
		I_3
		&= \sum_{K\in\mathcal{T}}\sum_{\sigma\in\mathcal{E}_K}\left(\int_\sigma \widehat{u_i}^j\nabla \widehat{u_0}^j\cdot n_{K,\sigma}(\M_Kf-f)+\left(\widehat{u_0}^j\nabla \widehat{u_i}^j\right)\cdot n_{K,\sigma} f\, dS-u_{i,\sigma}^jF_{\sigma}^K(u_0^j)\M_Kf\right) \\
		&= \sum_{K\in\mathcal{T}}\sum_{\sigma\in\mathcal{E}_K}\left(\int_\sigma \left(\frac12\jump{\widehat{u_i}^j\nabla \widehat{u_0}^j}+\widehat{u_i}^j\dc{\nabla \widehat{u_0}^j}\cdot n_{K,\sigma}\right)(\M_Kf-f)+\frac12\jump{\widehat{u_0}^j\nabla \widehat{u_i}^j} f\, dS-u_{i,\sigma}^jF_{\sigma}^K(u_0^j)\M_Kf\right) \\
		 &= \sum_{K\in\mathcal{T}}\left(\frac12\sum_{\sigma\in\mathcal{E}_K\setminus\mathcal{E}^D}\int_\sigma \llbracket \widehat{u_0}^j\nabla \widehat{u_i}^j-\widehat{u_i}^j\nabla \widehat{u_0}^j\rrbracket(f-\M_K f)\,dS\right.\\
		&\phantom{=}\left.+\sum_{\substack{\sigma\in\mathcal{E}_K\setminus\mathcal{E}^N}}\int_\sigma \left(u_{i,\sigma}^j\frac{F_{\sigma}^K(u_0^j)}{\abs{\sigma}} -\widehat{u_i}^j\dc{\nabla \widehat{u_0}^j}\cdot n_{K,\sigma}\right)(f-\M_Kf)\,dS\right).
	\end{align*}
	With Hölder and the trace inequality (Theorem \ref{pFaceInterpolation}) we arrive at
	\begin{multline}\label{eqn:bound_I_3_4}
		I_3\leq C_{\text{app},2}\frac12\sum_{K\in\mathcal{T}}h_K^\frac12\left(\sum_{\sigma\in\mathcal{E}_K\setminus\mathcal{E}^D}\norm{\llbracket \widehat{u_0}^j\nabla \widehat{u_i}^j-\widehat{u_i}^j\nabla \widehat{u_0}^j\rrbracket}_\sigma\right.\\
		\left.+\sum_{\sigma\in\mathcal{E}_K\setminus\mathcal{E}^N}\norm{u_{i,\sigma}\frac{F_{\sigma}^K(u_0^j)}{\abs{\sigma}}-\widehat{u_i}^j\dc{\nabla \widehat{u_0}^j}\cdot n_{K,\sigma}}_\sigma\right)\norm{\nabla f}_K
	\end{multline}
	Combining \eqref{eqn:bound_I_1},\eqref{eqn:bound_I_2} and \eqref{eqn:bound_I_3_4} we can estimate with the discrete Hölder inequality
	\begin{align*}
		I\leq \left(R_{S,i}+R_{R,i}\right)\norm{\nabla f}_\Omega
	\end{align*}

	We can now proceed with the temporal part $II$ using Hölder's inequality
	\begin{multline*}
		\int_\Omega\left( \widehat{u_0}\nabla \widehat{u_i}-\widehat{u_i}\nabla \widehat{u_0}-\widehat{u_0}^j\nabla \widehat{u_i}^j+\widehat{u_i}^j\nabla \widehat{u_0}^j\right)\cdot\nabla f\,dx\\
		\leq \norm{\widehat{u_0}^j\nabla \widehat{u_i}^j-\widehat{u_i}^j\nabla \widehat{u_0}^j-\widehat{u_0}^{j-1}\nabla \widehat{u_i}^{j-1}+\widehat{u_i}^{j-1}\nabla \widehat{u_0}^{j-1}}_\Omega\norm{\nabla f}_\Omega
	\end{multline*}
\end{proof}
\subsection{Estimator}
We can now state a reliable a posteriori error bound for the difference $u_i-v_i$.
For this we use the stability framework from Section \ref{section:abstract_error}, the a posteriori error bounds of the diffusion equation from Section \ref{section:heat} and the residual bound from Theorem \ref{thm:Residual_bound}.
\begin{thm}\label{thm:error_estimator}
	Let Condition \ref{conditions} (C1)-(C7) hold.
	Let $(v_0,\dots,v_n)$ be a weak solution to the reduced model \eqref{eqn:reduced} similar to Theorem \ref{thm:existence} and $(\widehat{u_0},\dots,\widehat{u_n})$ the Morley type reconstruction of a finite volume solution discussed in Section \ref{section:scheme}.
	The difference $\widehat{u_i}-v_i$ for $i=1,\dots,n$ is bounded by
	\begin{multline*}
		\max_{t\in[0,T]}\norm{\widehat{u_i}-v_i}_{\Omega}^2+\norm{\sqrt{v_0}\nabla (\widehat{u_i}-v_i)}_{[0,T]\times\Omega}^2\\
		\leq \left(2\norm{\widehat{u_i}^0-v_i^0}_{\Omega}+\frac{12}\gamma\norm{\widehat{u_0}^0-v_0^0}_\Omega^2+\frac{12}\gamma\sum_{j=0}^J\tau_j\left(\left(\eta_{R,i}^j\right)^2+\left(\eta_{2}^j\right)^2\right)+\frac{12}{\gamma}\norm{\nabla \widehat{u_i}}_{[0,T]\times\Omega}^2\left(\eta_\infty^J\right)^2\right)\\
		\times\exp\left(2C_G^{\frac{2}{1-\theta}}(1+C_{F,2,\Gamma_D}^2)^{\frac{\theta}{1-\theta}}\frac{\mu}{\gamma^{\frac{1+\theta}{1-\theta}}}\norm{\nabla \widehat{u_0}}_{X(q)}^{\frac2{1-\theta}}\right).
	\end{multline*}
	with $\theta=\frac{d}{2}-\frac{d}p$ and $\mu=\frac{1-\theta}{2}\left(\frac{1}{2(1+\theta}\right)^{\frac{1+\theta}{1-\theta}}$
	\begin{align*}
		\eta_\infty^J&:=\norm{v_0^0-\widehat{u_0}^0}_{L^\infty(\Omega)}+\sum_{j=1}^J\eta_{T,\infty}^j+\sum_{j=1}^J\dot{\eta}_{S,q}^jC_{\text{Green},p}+\max_{0\leq j\leq J} \eta_{S,q}^jC_{\text{Green},p}\\
		\eta_2^j&:=\eta_{S,2}^j+\norm{\nabla\left(\widehat{u_0}^j-\widehat{u_0}^{j-1}\right)}_{\Omega}\\
		\eta_{R,i}^j &:= R_{S,i}^j+R_{T,i}^j+R_{R,i}^j,
	\end{align*}
	where $\eta^j_{T,\infty},\dot{\eta}_{S,q}^j$ were defined in Lemma \ref{lemma:ParabolicBound}, $\eta_{S,q}^j$ was defined in Lemma \ref{lemma:EllipticBound} and $R_{S,i}^j,R_{T,i}^j, R_{R,i}^j$ were defined in Theorem \ref{thm:Residual_bound}.
\end{thm}
\begin{proof}
	Theorem \ref{thm:abstract_error} yields
	\begin{multline*}
		\max_{t\in[0,T]} \norm{\widehat{u_i}-v_i}_\Omega^2+\int_0^{T}\norm{\sqrt{v_0}\nabla(\widehat{u_i}-v_i)}_\Omega^2\,dt
		\leq\\
		\left(2\norm{\widehat{u_i}^0-v_i^0}^2_{\Omega}
		+\frac{12}\gamma\norm{v_0-\widehat{u_0}}_{L^\infty([0,T]\times\Omega)}^2\norm{\nabla \widehat{u_i}}_{[0,T]\times\Omega}^2
		+\frac{12}\gamma\norm{\nabla(\widehat{u_0}-v_0)}_{[0,T]\times\Omega}^2
		+\frac{12}\gamma \norm{R_i}_{*}^2 \right)\\
		\exp\left(2C_G^{\frac{2}{1-\theta}}(1+C_{F,2,\Gamma_D}^2)^{\frac{\theta}{1-\theta}}\frac{\mu}{\gamma^{\frac{1+\theta}{1-\theta}}}\norm{\nabla \widehat{u_0}}_{X(q)}^{\frac{2}{1-\theta}}\right).
	\end{multline*}
	We now can use the estimates from Theorem \ref{thm:Residual_bound}, Theorem \ref{thm:H1_heat_bound} and Theorem \ref{thm:max_heat_bound} to infer the claim.
\end{proof}
\subsection{Error estimator for the general model}\label{subsec:general_residual}
Similar to the reduced model, given the stability result, Theorem \ref{thm:abstract_error}, we now only need bounds on the residuals $R_0,\dots,R_n$ and $R_\Phi$ in the $H^1_D(\Omega)'$ norm.
First we derive a bound for the residual $R_\Phi$.
Using integration by parts and the properties of the reconstruction we can bound the residual by
\begin{equation}\label{eqn:residual_bound_phi}
\begin{aligned}
	\langle R_\Phi,f\rangle
	&= \left( \nabla \widehat{\Phi},\nabla f\right)-\frac{z}{\lambda^2}\left( 1-\widehat{u_0},f\right)\\
	&= -\sum_{K\in\mathcal{T}}\int_K\left(\frac{z}{\lambda^2}(1-\widehat{u_0})+\Delta \widehat{\Phi}\right)f\,dx-\sum_{\sigma\in\mathcal{E}}\int_\sigma\llbracket\nabla \widehat{\Phi}\rrbracket f\,dS
\end{aligned}
\end{equation}
We use the same technique as in \cite[Corollary 5.4]{nicaise2005} to arrive at
\begin{align*}
	\langle R_\Phi,f\rangle
	&= -\sum_{K\in\mathcal{T}}\int_K\left(\frac{z}{\lambda^2}(1-u_{0,h})+\Delta \widehat{\Phi}\right)(f-\M_Kf)\,dx-\sum_{\sigma\in\mathcal{E}}\int_\sigma\llbracket\nabla \widehat{\Phi}\rrbracket (f-\M_Kf)\,dS+\int_\Omega\frac{z}{\lambda^2}(\widehat{u_0}-u_{0,h})f\,dx\\
	&	\leq \left(\sqrt{2}\left(\sum_{K\in\mathcal{T}}C_{P,2}^2h_K^2\norm{\frac{z}{\lambda^2}(1-u_{0,h})+\Delta \widehat{\Phi}}_K^2+\frac{C_{\text{app},2}^2h_KN_\partial}{4}\sum_{\sigma\in\mathcal{E}_K\setminus\mathcal{E}^D} \norm{\llbracket\nabla \widehat{\Phi}\rrbracket}_\sigma^2\right)^\frac12 \right.\\
	&\left.\quad+\frac{\abs{z}}{\lambda^2}C_{F,2,\Omega}\norm{\widehat{u_0}-u_{0,h}}_\Omega\right)\norm{\nabla f}_\Omega.
\end{align*}
We also can bound the residual $R_0$ using the same technique as in Section \ref{Section:Residual_bound}
\begin{align*}
	\norm{R_0(t)}_{H^1_D(\Omega)'}
	&\leq R_{T,0}^j+R_{S,0}^j+R_{R,0}^j =:\eta_{R,0}^j\quad\forall t\in[t_{j-1},t_j]
\end{align*}
with the temporal, spacial and residual bounds given by
\begin{equation}\label{eqn:residual_bound_0}
\begin{aligned}
	R_{T,0}^j &= \norm{\nabla \widehat{u_0}^j-\widehat{u_0}^j(1-\widehat{u_0}^j)\beta z\nabla\Phi^j-\nabla \widehat{u_0}^{j-1}+\widehat{u_0}^{j-1}(1-\widehat{u_0}^{j-1})\beta z\nabla\Phi^{j-1}}_\Omega\\
	R_{S,0}^j &= \sqrt{2}\left(\sum_{K\in\mathcal{T}}\left(C_{P,2}^2h_K^2\norm{\partial_t u_{0,h}-\operatorname{div}\left(\nabla \widehat{u_0}^j-\widehat{u_0}^j(1-\widehat{u_0}^j)\beta z\nabla \widehat{\Phi}^j\right)}_K^2\right.\right.\\
	&\left.\left.+\frac{N_\partial C_{\text{app},2}^2}{4}\sum_{\sigma\in\mathcal{E}\setminus\mathcal{E}^D}h_K\norm{\llbracket \nabla \widehat{u_0}^j-\widehat{u_0}^j(1-\widehat{u_0}^j)\nabla \widehat{\Phi}^j\rrbracket}_\sigma^2\right)\right)^\frac12\\
	R_{R,0}^j &=C_{F,2,\Gamma_N}\norm{\partial_t (\widehat{u_0}-u_{0,h})}_{\Omega}\\
	&+\frac{C_{\text{app},2}N_\partial^\frac12}{2}\left(\sum_{K\in\mathcal{T}}\sum_{\sigma\in\mathcal{E}\setminus\mathcal{E}^N}h_K\norm{\widehat{u_0}^j\beta z(1-\widehat{u_0}^j)\left\{\!\!\left\{\nabla \widehat{\Phi}^j\right\}\!\!\right\}\cdot n_{K,\sigma}-u_{0,\sigma}^j\beta z\left(\sum_{i=1}^nu_{i,\sigma}^j\right)\frac{F_{\sigma}^K(\Phi^j)}{\abs{\sigma}}}_{\sigma}^2\right)^\frac12.
\end{aligned}
\end{equation}
With \eqref{equations} we notice that $\sum_{i=1}^n u_{i,\sigma}^j$ is an approximation of $1-u_0$ on the edge $\sigma$.
The $R_{0,R}^j$ in \eqref{eqn:residual_bound_0} is also consistent with \eqref{eqn:scheme_0}.
Hence, we expect that the second term in $R_{0,R}^j$ is small.\\
The residuals $R_i$ can be bounded similar to Theorem \ref{thm:Residual_bound}
\begin{align*}
	\norm{R_i(t)}_{H^1_D(\Omega)'} &\leq R_{S,i}^j+R_{T,i}^j+R_{R,i}^j =:\eta_{R,i}^j\quad\forall t\in[t_{j-1},t_j]
\end{align*}
where the spacial bound $R_{S,i}^j$, temporal bound $R_{T,i}^j$ and reconstruction bound $R_{R,i}^j$ are given by
\begin{equation}\label{eqn:residual_bound_i}
	\begin{aligned}
		R_{S,i}^j &= \sqrt{2}\left(\sum_{K\in\mathcal{T}}\left(C_{P,2}^2h_K^2\norm{\partial_t u_{i,h}-\operatorname{div}\left(\widehat{u_0}^j\nabla \widehat{u_i}^j-\widehat{u_i}^j\nabla \widehat{u_0}^j+\widehat{u_i}^j\widehat{u_0}^j\beta z\nabla \widehat{\Phi}^j\right)}_K^2\right.\right.\\
		&+\left.\left.\frac{N_\partial C_{\text{app},2}^2}{4}\sum_{\sigma\in\mathcal{E}\setminus\mathcal{E}^D}h_K\norm{\llbracket \widehat{u_0}^j\nabla \widehat{u_i}^j-\widehat{u_i}^j\nabla \widehat{u_0}^j+\widehat{u_i}^j\widehat{u_0}^j\beta z\nabla \widehat{\Phi}^j\rrbracket}_\sigma^2\right)\right)^\frac12\\
		R_{T,i}^j &= \left\|\widehat{u_0}^j\nabla \widehat{u_i}^j-\widehat{u_i}^j\nabla \widehat{u_0}^j+\widehat{u_i}^j\widehat{u_0}^j\beta z\nabla \widehat{\Phi}^j\right.\\
		&\left.-\widehat{u_0}^{j-1}\nabla \widehat{u_i}^{j-1}+\widehat{u_i}^{j-1}\nabla \widehat{u_0}^{j-1}-\widehat{u_i}^{j-1}\widehat{u_0}^{j-1}\beta z\nabla \widehat{\Phi}^{j-1}\right\|_\Omega\\
		R_{R,i}^j &=C_{P,2,\Omega}\norm{\partial_t (\widehat{u_i}-u_{i,h})}_{\Omega}\\
		&+C_{\text{app},2}N_\partial^\frac12\left(\sum_{\sigma\in\mathcal{E}\setminus\mathcal{E}^N}h_K\norm{\widehat{u_i}^j\left\{\!\!\!\left\{\nabla \widehat{u_0}^j-\beta z\widehat{u_0}^j\nabla\Phi^j\right\}\!\!\!\right\}\cdot n_{K,\sigma}-u_{i,\sigma}\frac{F_{\sigma}^K(u_0^j)-\beta z F_{\sigma}^K(\Phi^j)}{\abs{\sigma}}}_{\sigma}^2\right)^\frac12.
	\end{aligned}
\end{equation}
Notice that the residual bound has the same structure as the residual bound in Theorem \ref{thm:Residual_bound}.
The only difference is, that the convective part has changed from $-\widehat{u_i}\nabla \widehat{u_0}$ for the reduced model to $-\widehat{u_i}(\nabla \widehat{u_0}-\beta z\widehat{u_0}\nabla \Phi)$.
\begin{proof}[Proof of Theorem \ref{thm:main_estimator}]
	The claims follow from Theorem \ref{thm:abstract_general} and the bounds on the residuals from \eqref{eqn:residual_bound_phi}-\eqref{eqn:residual_bound_i}.
\end{proof}
\begin{remark}\label{remark:constants}
	All the constants appearing in the estimator in Theorem \ref{thm:main_estimator} except for $\gamma$ are computable for convex polygonal domains.
	The following bound on the Poincar\'e-Friedrichs constant $C_{F,2,\Gamma_D}$ for convex domains can be found in \cite{pauly2020}.
	\begin{align*}
		C_{P,2} \leq C_{F,2,\Gamma_D} \leq \frac{\operatorname{diam}(\Omega)}{\pi}.
	\end{align*}
	A bound on the Sobolev constant $C_S$ for a finite union of convex domains can be found in \cite{mizuguchi2017}.
	A bound for the constant in the Gagliardo-Nirenberg inequality used here can be found in Theorem \ref{thm:Nirenberg}.
\end{remark}
\section{Numerical Results}\label{section:numerics}
We now show the results of numerical experiments to validate the optimal scaling behavior of the error estimators from Theorem \ref{thm:error_estimator} and Theorem \ref{thm:main_estimator}.
The implementation is done in Python.
We define the estimated order of convergence (EOC) of two sequences $a(i)$ and $h(i)$ by
\begin{align*}
	\text{EOC}(a,h;i):=\frac{\log(a(i+1)/a(i))}{\log(h(i+1)/h(i))}.
\end{align*}
We also look at the effectivity index
\begin{align*}
	\operatorname{EI}_i &:=
	\begin{dcases}
		\frac{\sqrt{\max_{t\in[0,T]}\norm{v_0(t)-\widehat{u_0}(t)}_{\Omega}^2+\norm{\nabla(v_0-\widehat{u_0})}_{[0,T]\times\Omega}^2}}{\eta_2^J} &\text{for } i=0,\\
		\frac{\sqrt{\max_{t\in[0,T]}\norm{v_i(t)-\widehat{u_i}(t)}_{\Omega}^2+\norm{\sqrt{v_0}\nabla(v_i-\widehat{u_i})}_{[0,T]\times\Omega}^2}}{\eta_2^{i,J}} &\text{for } i=1,\dots,n.
	\end{dcases}\\
	 \operatorname{EI}_\Phi &:=\frac{\norm{\nabla (\widehat{\Phi}-\Psi)}_{[0,T]\times\Omega}}{\eta_\Phi^J}
\end{align*}
where $\eta_2^{i,J}$ and $\eta_2^J$ are the estimators defined in Theorem \ref{thm:error_estimator} for the reduced model and from Theorem \ref{thm:main_estimator}.
\subsection{The reduced model}
In what follows, we choose $\Omega=[0,1]^2$ with Dirichlet boundary $\Gamma_D:=\{0,1\}\times[0,1]$ and Neumann boundary $\Gamma_N:=(0,1)\times\{0,1\}$.
Furthermore, we set $T=1$ and consider the time intervall $[0,1]$.
We use
\begin{align*}
	v_1(t,x,y) &= 0.1+0.1x+tx(1-x)e^{-20x^2}\\
	v_2(t,x,y) &= 0.1-0.05x-0.5(1-\cos(t))x(1-x)\sin(x)\\
	v_3(t,x,y) &= 0.2+0.1x+tx(1-x)e^{-20(x-0.5)^2}.
\end{align*}
as a manufacured soluion to \eqref{eqn:reduced}.
With initial and boundary condition chosen accordingly.
With this choice, we arrive at $\gamma=0.3$.
We here fix $q=42$ and $\tilde{q}=2.1$ and the other constants are then given by
\begin{align}\label{eqn:constants}
	\mu\leq 1.08,\quad
	C_G\leq 1.02,\quad
	C_S\leq 12.02
\end{align}
with Theorem \ref{thm:Nirenberg} and \cite[Table 2]{mizuguchi2017} and the Poincar\'e and Poincar\'e-Friedrichs constant given in Remark \ref{remark:constants}.
For the maximum norm estimator $\eta_\infty^J$ from Theorem \ref{thm:max_heat_bound} we set $p^*=1$ and $C_{\text{Green},1}=1$, since this constant is not accessible (see Remark \ref{remark:ass}).
A sequence of approximate solutions is obtained from the scheme and reconstruction described in Section \ref{section:scheme} using mesh width $h=2^{-i-1}$ and time step size $\tau=2^{-i-1}$ for $i=0,\dots,5$.\\
In Table \ref{reduced_eoc_table} we see a linear convergence of the estimators $\eta_2^J,\eta_\infty^J$ and $\eta_2^{1,J}$.
The estimators for the other species are similar to $\eta_2^{1,J}$ and are therefore omitted.
The effectivity indices indicate that the estimator $\eta_2^{1,J}$ overestimates the error by a large factor but the estimator scales with the same order as the true error.
\begin{table}[h]
	\centering
\begin{tabular}{|l|l|l|l|l|l|l|l|l|}
	\hline
	i& $\eta_2^J$&EOC &$\eta_\infty^J$&EOC&$\eta_2^{1,J}$&EOC&$\operatorname{EI}_0$&$\operatorname{EI}_1$\\\hline
	0&1.458&1.1&18.14&0.98&118.1&1.2&0.16&0.00076\\
	1&0.6916&1.0&9.194&0.53&51.55&0.5&0.15&0.00084\\
	2&0.337&1.0&6.359&0.79&36.36&0.68&0.15&0.00055\\
	3&0.1655&1.0&3.669&0.96&22.67&0.91&0.15&0.00043\\
	4&0.0819&1.0&1.891&0.99&12.09&0.98&0.15&0.0004\\
	5&0.04072&-&0.9507&-&6.148&-&0.15&0.00039\\
	\hline
\end{tabular}
\caption{Estimators, EOC and effectivity index for the reduced model.}
\label{reduced_eoc_table}
\end{table}
\subsection{The general model}
In what follows, we choose $\Omega=[0,1]^2$ with Dirichlet boundary $\Gamma_D:=\{0,1\}\times[0,1]$ and Neumann boundary $\Gamma_N:=(0,1)\times\{0,1\}$.
We use
\begin{align*}
	v_1(t,x,y)&:=0.1+t^2\theta(y)x(1-x)\exp(-100(x-0.5)^2)\\
	v_2(t,x,y)&:=0.2x+0.1(1-x)-\frac12\sin(t)\theta(y))x(1-x)\cos(x)\\
	v_3(t,x,y)&:=0.2-0.1x-0.55tx(1-x)\theta(y)\exp(-100(x-0.5)^2)\\
	\Psi(t,x,y)&:=\frac{25}{3}(x-1)x(5x^2-2.6-2.4+t(x^2-x))\theta(y)
\end{align*}
with
\begin{align*}
	\theta(y) := \frac{y^4}3-2\frac{y^3}3+\frac{y^2}{3}+\frac{47}{48}
\end{align*}
as a manufacured solution to \eqref{eqn:general} with constants $\beta=z=\lambda=1$.
We again choose $q=42$ and $\tilde{q}=2.1$ and the same constants as in \eqref{eqn:constants}.
A sequence of approximate solutions is obtained again by using the mesh width $h=2^{-i-1}$ and time step size $\tau=2^{-i-1}$ for $i=0,\dots,5$.
In Table \ref{table:eoc_general_model} we see a linear convergence of the estimators $\eta_2^J,\eta_2^{1,J}$ and $\eta_\Phi^{J}$.
The estimators for the other species are similar to $\eta_2^{1,J}$ and are therefore omitted.
We see that the estimators $\eta_2^J,\eta_2^{1,J}$ for the general model are much larger than the estimator for the reduced model (see Table \ref{reduced_eoc_table}).
One reason for this is that due to the choice of example the norm appearing in the exponential term is much larger, i.e. $\norm{F}_{X(q)}^{\frac{2}{1-\theta}}\approx 4.7$ for the general model compared to $\norm{\nabla \widehat{u_0}}_{X(q)}^{\frac{2}{1-\theta}}\approx 0.4$ for the reduced model.
Furthermore, the maximum norm estimator for the solvent concentration in the reduced model enables us to use the $L^2([0,T]\times\Omega)$ norm of $\nabla \widehat{u_1}$ instead of the $L^\infty(0,T;L^q(\Omega))$ norm in the general model, also contributing to the smaller effectivity index in the second experiment.
The estimator of the error of solvent concentration $\eta_2^J$ for the general model has an additional exponential term that is large for this example.
\begin{table}[h]
	\centering
\begin{tabular}{|l|l|l|l|l|l|l|l|l|l|}
	\hline
	i& $\eta_2^J$&EOC &$\eta_2^{1,J}$&EOC&$\eta_\Phi^{J}$&EOC&$\operatorname{EI}_0$&$\operatorname{EI}_1$&$\operatorname{EI}_\Phi$\\\hline
	0&2.667e+07&-1.2&3.811e+11&-1.8&8.489e+06&-1.2&5.8e-09&8.7e-13&8.9e-08\\
	1&6.157e+07&0.46&1.292e+12&0.3&1.96e+07&0.46&1.4e-09&1.5e-13&1.9e-08\\
	2&4.484e+07&0.9&1.047e+12&0.86&1.427e+07&0.9&9e-10&8.3e-14&1.2e-08\\
	3&2.408e+07&0.98&5.784e+11&0.97&7.666e+06&0.98&8.1e-10&7.3e-14&1.1e-08\\
	4&1.218e+07&1.0&2.946e+11&1.0&3.878e+06&1.0&7.9e-10&7.1e-14&1.1e-08\\
	5&6.091e+06&-&1.475e+11&-&1.939e+06&-&7.9e-10&7e-14&1.1e-08\\
	\hline
\end{tabular}
\caption{Estimators, EOC and effectivity index for the general model.}
\label{table:eoc_general_model}
\end{table}\\
We see that the error estimators for the reduced model and the general model converge linearly.
Since the reconstruction presented in Section \ref{section:scheme} is mostly made up of piecewise linear polynomials and the error estimates bound a (weighted) $L^2(0,T;L^2(\Omega))$-norm of the difference of the gradients of the reconstruction and the weak solution, the linear convergence of the error estimators is expected.
	The research code associated with this article are available in the git-repository \url{https://git-ce.rwth-aachen.de/arne.berrens/a-posteriori-error-control-for-a-finite-volume-scheme-for-a-cross-diffusion-model-of-ion-transport}\cite{berrens_repo}.
	\paragraph{Funding}
	The research of JG was supported by the Deutsche Forschungsgemeinschaft (DFG, German Research Foundation) - SPP 2410 Hyperbolic Balance Laws in Fluid Mechanics: Complexity, Scales, Randomness (CoScaRa) within the project 525877563 (A posteriori error estimators for statistical solutions of barotropic Navier-Stokes equations). JG also acknowledges support by the German Science Foundation (DFG) via grant TRR 154 (Mathematical modelling, simulation and optimization using the example of gas networks), sub-project C05 (Project 239904186).
	\printbibliography[heading=bibintoc]
\end{document}